\documentclass
[a4paper,
11pt,
twoside
]{article}

\usepackage[T1]{fontenc}
\usepackage{ae}

\usepackage[left=1in,right=1in,top=1.2in,bottom=0.9in, headheight=15pt]{geometry}

\usepackage{authblk} 

\usepackage{amsmath}
\usepackage{amssymb}
\usepackage{amsthm}

 \usepackage{charter,courier} 
 \usepackage{chemarrow}

\usepackage[figuresright]{rotating} 
\usepackage{fancyhdr}
\usepackage{empheq}
\usepackage{latexsym}
\usepackage{lscape}

\usepackage{makeidx}
\makeindex

\usepackage{mleftright} 

\usepackage[nottoc]{tocbibind}

\usepackage[intoc]{nomencl}
\makenomenclature
\usepackage{stmaryrd}

\usepackage{textcomp}
\usepackage{tikz}

\usepackage{wasysym}

\usepackage{xcolor}
\usepackage{color}
\usepackage{colortbl}

\usepackage{yhmath}
\usepackage{graphicx}

\usepackage[all]{xy}

\definecolor{colComments}{rgb}{1,0,0}

\theoremstyle{plain}  

\swapnumbers 
\numberwithin{equation}{section} 
\newtheorem{Definition}[equation]{Definition} 
\newtheorem{Definition/Lemma}[equation]{Definition/Lemma}

\newtheorem{Lemma}[equation]{Lemma}
\newtheorem{Theorem}[equation]{Theorem}
\newtheorem{Proposition}[equation]{Proposition}
\newtheorem{Corollary}[equation]{Corollary}
\newtheorem{Remark}[equation]{Remark}

\newtheorem{Notation}[equation]{Notation}
\newtheorem{Notation/Lemma}[equation]{Notation/Lemma}

\title{Character tables of Sylow $p$-subgroups
         of the \\ Steinberg triality groups ${^3}D_4(q^3)$}
\author{Yujiao Sun}
\affil{\footnotesize School of Mathematics and Statistics,
         Beijing Institute of Technology,\\
      Beijing 100081, PR China}
\affil{\footnotesize E-mail:
           yujiaosun@bit.edu.cn}
\date{}

\begin{document}
\maketitle

\begin{abstract}
We determine the character tables of Sylow $p$-subgroups $U$
of the Steinberg triality groups ${^3}D_4(q^3)$,
where $q$ is a power of an odd prime $p$.
\end{abstract}





\section{Introduction}
Let $p$ be a fixed prime,
$\mathbb{N}^*$ the set of positive integers,
$q:=p^k$ with a fixed $k\in \mathbb{N}^*$,
$\mathbb{F}_q$ the finite field with $q$ elements
and $U_n(q)$ ($n\in \mathbb{N}^*$)
the group of upper unitriangular $n\times n$-matrices
(i.e. upper triangular matrices with $1$s on the main diagonal)
with entries in $\mathbb{F}_q$.
Then $U_n(q)$ is a Sylow $p$-subgroup of the general linear group $GL_n(q)$
of invertible $n\times n$-matrices over $\mathbb{F}_q$.

G. Higman \cite{hig} formulated a longstanding conjecture:
for a fixed $n$, the number of conjugacy classes of $U_n(q)$
is determined by a polynomial in $q$ with integral coefficients depending on $n$.
G. Lehrer \cite{leh} refined Higman's conjecture:
the number of irreducible characters of $U_n(q)$
of degree $q^c$ ($c\in \mathbb{N}$)
is an integer polynomial in $q$.
I. M. Isaacs \cite{isa1} proved that
the degrees of complex irreducible characters
of $\mathbb{F}_q$-algebra groups are powers of $q$.
Then I. M. Isaacs \cite{isa2} gave a strengthened form of Lehrer's conjecture:
the number of irreducible characters of degree $q^c$
is some polynomial in $(q-1)$ with non-negative integeral coefficients.
A. Vera-L\'{o}pez and J. M. Arregi  \cite{vla}
proved Higman's conjecture for $n\leq 13$.
Recently, I. Pak and A. Soffer \cite{ps} verified Higman's conjecture for $n\leq 16$.

It is natural to consider Higman's conjecture,
Lehrer's conjecture and Isaacs' conjecture for
the Sylow $p$-subgroups of other finite groups of Lie
type.
Let $G(q)$ be a finite group of Lie type,
$U(q)$ a Sylow $p$-subgroup of $G(q)$,
$k(U(q))$ the number of conjugacy classes,
$\#\mathrm{Irr}(U(q))$ the number of all complex irreducible characters,
$\#\mathrm{Irr}(U(q), q^c)$ the number of complex irreducible characters of degree $q^c$,
$\#M(U(q))$ the number of pairwise orthogonal irreducible constituents of
the regular $U(q)$-module $\mathbb{C}U(q)$.
S. M. Goodwin,  P. Mosch and G. R{\"o}hrle
\cite{GR2009, GMR2014, GMR2016}
obtained an algorithm
and calculated  $k(U(q))$ for $U(q)$ of rank at most 8, except $E_8$.

For the Sylow $p$-subgroup $U(q)$ of the Chevalley group $D_4(q)$ of type $D_4$,
F. Himstedt, T. Le and K. Magaard \cite{HLMD42011}
determined the complex irreducible characters
and calculated $\#\mathrm{Irr}(U(q), q^c)$.
M. Jedlitschky \cite[Appendix A.3]{Markus1}
decomposed the regular module $\mathbb{C}U(q)$ with $p\neq 2$ into irreducible constituents,
and obtained $\#\mathrm{Irr}(U(q), q^c)$ and $\#\mathrm{Irr}(U(q))$.
S. M. Goodwin,  T. Le and K. Magaard
\cite{GLM2017}
constructed the {\it generic character table} of $U(q)$.
Higman's conjecture,
Lehrer's conjecture and Isaacs' conjecture are true for $U(q)$.

For the Sylow $p$-subgroup $U(q)$ $(p>3)$ of the Chevalley group $G_2(q)$ of type $G_2$,
F. Himstedt, T. Le and K. Magaard \cite{HLM2016}
determined most irreducible characters of $U(q)$
by parameterizing {\it midafis} (minimal degree almost faithful irreducible characters).

S. M. Goodwin, T. Le, K. Magaard and A. Paolini \cite{GLMP2016}
parameterized the irreducible characters of
the Sylow $p$-subgroup $U(q)$ $(p>2)$ of the Chevalley group $F_4(q)$ of type $F_4$.

Let $U$ be a Sylow $p$-subgroup of
the Steinberg triality group ${{^3D}_4}(q^3)$.
T. Le \cite{Le} constructed and counted all ordinary irreducible characters of $U$,
mainly using Clifford theory.

In this paper,
let $p$ be a fixed odd prime,
$q:=p^k$ a fixed power of $p$.
In Section \ref{sec:Sylow p-subgroup 3D4},
we recall the construction of a Sylow $p$-subgroup $U$ of the Steinberg triality group ${^3}D_4(q^3)$.
In Section \ref{conjugacy classes},
we establish the conjugacy classes
and the explicit elements of the conjugacy classes
of the Sylow $p$-subgroup $U$.
Based on T. Le's work \cite{Le},
we determine the specific constructions
of all of the pairwise orthogonal complex irreducible characters of $U$
by Clifford's Theorem (see \cite[Section 11]{CR1})
in Section \ref{sec:irreducible characters}.
Finally, we further  calculate the values of the irreducible characters and obtain the character table of $U$
in Section \ref{sec:character table}.

If $p>2$, the classification of conjugacy classes of $U$ and the construction of irreducible character of $U$
are uniform.
If $p=2$, the Lemma \ref{zeta_u} does not hold, so the analysis is different.
Some results still hold when their proofs do not use this lemma.
For example, the conjugacy classes of  $x=x(0, 0, 0, t_4, t_5, t_6)\in U$
($t_4\in \mathbb{F}_{q^3},\ t_5, t_6\in  \mathbb{F}_q$) for $p=2$ are the same as these in Table \ref{table:conjugacy classes-3D4} for $p>2$,
but the conjugacy classes of $x_3(t_3^*)\in U$
($t_3^* \in \mathbb{F}_{q^3}^*$) for $p=2$ and these for $p>2$ are different
(see Remark \ref{conjugacy classes p=2}).
Similarly, the irreducible characters in the family $\mathfrak{F}_4$ of Proposition \ref{construction of irr. char. of 3D4}
do not hold when $p=2$. The construction of this case can be found in  \cite[3.3]{Le}.
The irreducible characters in the other cases for $p>2$ and for $p=2$ are the same
(see Remark \ref{irr. characters p=2}).

The main results in this paper are summarized in the following theorem:
\begin{Theorem}
Let ${^3}D_4(q^3)$ be the Steinberg triality group over $\mathbb{F}_{q^3}$ when $p>2$,
and let $U$ be a Sylow $p$-subgroup of ${^3}D_4(q^3)$.
The conjugacy classes of $U$ are shown in Proposition \ref{prop:conjugacy classes-3D4}.
The explicit constructions of irreducible characters of $U$ are given in Proposition \ref{construction of irr. char. of 3D4}.
In Proposition \ref{prop:character table-3D4}, the character table of $U$ is obtained.
\end{Theorem}

Here we fix some notation:
Let
$K$ a field,
$K^*$ the multiplicative group $K\backslash\{0\}$ of $K$,
$K^+$ the additive group of $K$,
$\mathbb{F}_q$ the finite field with $q$ elements,
$\mathbb{F}_{q^3}$ the finite field with $q^3$ elements,
$\mathbb{C}$ the complex field,
$\mathbb{Z}$ the set of all integers,
$\mathbb{N}$ the set $\{0,1,2, \dots \}$ of all non-negative integers,
$\mathbb{N}^*$ the set $\{1,2,\dots \}$ of all positive integers.


\section{Sylow $p$-subgroups $U$ of the Steinberg triality group ${^3}D_4(q^3)$}
\label{sec:Sylow p-subgroup 3D4}
In this section, we recall the root system of $D_4$,
the construction
of a Sylow $p$-subgroup $U$ of the Steinberg triality group ${^3}D_4(q^3)$,
and the commutator relations of $U$.
The main references are
\cite{Carter1} and \cite{Hum}.

Let $\mathcal{V}_4$ be a fixed Euclidean space with the usual orthonormal basis
$\varepsilon_1$, $\varepsilon_2$, $\varepsilon_3$, $\varepsilon_4$,
where the inner product $(\ ,\ )$ is the usual one.
The set
$
 \Phi_{D_4}=\{\pm\varepsilon_i\pm\varepsilon_j \mid 1\leq i< j\leq 4\}
$
is a root system of type ${D}_4$, and $(r, r)=2$ for all $r\in \Phi_{D_4}$.
The fundamental system of roots of the root system $\Phi_{D_4}$ is
$
 \Delta_{D_4}=\{\varepsilon_1-\varepsilon_2,\ \varepsilon_2-\varepsilon_3,\ \varepsilon_3-\varepsilon_4,\ \varepsilon_3+\varepsilon_4\}.
$
The positive system (relative to $\Delta_{D_4}$) of roots of $\Phi_{D_4}$ is
$
 \Phi_{D_4}^+:
        =\{\varepsilon_i\pm\varepsilon_j  \mid  1 \leq i< j\leq4\}.
$
Let
$
 r_1:=\varepsilon_1-\varepsilon_2,\ r_2:=\varepsilon_2-\varepsilon_3,
 \ r_3:=\varepsilon_3-\varepsilon_4,\ r_4:=\varepsilon_3+\varepsilon_4$,
then $\Delta_{D_4}=\{r_1, r_2, r_3, r_4\}$.

Let $r:=\sum_{i=1}^4{x_ir_i}\in \mathcal{V}_4$, $s:=\sum_{i=1}^4{y_ir_i}\in \mathcal{V}_4$,
then we write
$r\prec s$,
if $\sum_{i=1}^4{x_i}<\sum_{i=1}^4{y_i}$,
or if $\sum_{i=1}^4{x_i}= \sum_{i=1}^4{y_i}$ and
the first non-zero coefficient $x_i-y_i$ is positive.
If $r=\sum_{i=1}^4{x_ir_i}\in \Phi_{D_4}$,
the {height} of $r$ (relative to $\Delta_{D_4}$) is denoted by $\mathrm{ht}(r):=\sum_{i=1}^4{x_i}$.
Let
$
 r_5:= r_1+r_2,\
 r_6:= r_2+r_3,\
 r_7:= r_2+r_4,\
 r_8:= r_1+r_2+r_3,\
 r_{9}:= r_1+r_2+r_4,\
 r_{10}:= r_2+r_3+r_4,\
 r_{11}:= r_1+r_2+r_3+r_4,\
 r_{12}:= r_1+2r_2+r_3+r_4$,
then
$\Phi_{D_4}^+=\{r_i \mid  i=1,2,\dots, 12\}$.
We get the total order of $\Phi_{D_4}^+$ as follows:
\begin{align*}
 0  \prec r_1 \prec r_2 \prec r_3 \prec r_4
    \prec r_5 \prec r_6 \prec r_7
    \prec r_8 \prec r_9 \prec r_{10}
    \prec r_{11}
    \prec r_{12}.
\end{align*}
Let $\mathcal{L}$ be a complex simple Lie algebra of type $D_4$
with
a Cartan subalgebra $\mathcal{H}$,
a Cartan decomposition
$\mathcal{L}=\mathcal{H}\oplus \sum_{r\in\Phi_{D_4}}\mathcal{L}_r,$
and a {Chevalley basis}
$\{h_r \mid r\in \Delta_{D_4}\}\cup \{e_r \mid r\in \Phi_{D_4}\}$,
where $e_r\in \mathcal{L}_r$ for all $r\in \Phi_{D_4}$.
Then $[e_r,e_s]=N_{r,s}e_{r+s}$ for all $r, s \in \Phi_{D_4}$,
where $N_{r,s}=\pm(n_{r,s}+1)$ and $n_{r,s}$ is the greatest integer for which $s-n_{r,s}r \in \Phi_{D_4}$.

We choose the signs of the structure constants $N_{r,s}$ for extra special pairs $(r,\ s)$
(see \cite[Page 58]{Carter1}) as follows:
\begin{alignat*}{2}
& -N_{r_1,r_2}=N_{r_2,r_3}=N_{r_2,r_4}=1, & \qquad
& N_{r_1,\ r_2+r_3+r_4}=1,\\
& N_{r_1,r_2+r_3}=N_{r_1,r_2+r_4}=N_{r_2,r_3+r_4}=1, & \qquad
& N_{r_2,\ r_1+r_2+r_3+r_4}=1.
\end{alignat*}
Then the structure constants of the other special pairs
(see \cite[Page 58]{Carter1})
of $r_i\in \Phi_{D_4}^+\backslash \Delta_{D_4}$ are 1.

Let $x\in \mathcal{L}$.
We define a linear map by
$\mathrm{ad}{\,x}:\mathcal{L} \to  \mathcal{L}: y\mapsto [x,y]:=xy-yx$.
Let $x_{r}(t):=\exp ({t{\ }\mathrm{ad}e_r})$ for all $r\in \Phi_{D_4}$,
$t \in \mathbb{C}$.
Let $K$ be a field, then the Chevalley group of $\mathcal{L}$ over $K$ is defined by
$D_4(K):=\langle x_r(t)  \mid  r\in \Phi_{D_4},  t \in K\rangle$.
The Chevalley group of $\mathcal{L}$ over $\mathbb{F}_q$ is also denoted by $D_4(q)$.
Let $r\in \Phi_{D_4}$, then the root subgroup is denoted by
$X_r:=\langle x_r(t)  \mid  t \in K\rangle$.
The subgroup $\langle x_r(t)  \mid  r\in \Phi_{D_4}^+,  t \in \mathbb{F}_q\rangle$
is a Sylow $p$-subgroup of $D_4(q)$.

Let $\sigma$ be a non-trivial symmetry of the Dynkin diagram of $\mathcal{L}_{D_4}$
sending $r_1$ to $r_3$, $r_3$ to $r_4$, $r_4$ to $r_1$, and fixing $r_2$,
$\rho$ be a linear transformation of $\mathcal{V}_4$ into itself
arising from $\sigma$,
then $\rho(\Phi_{D_4})=\Phi_{D_4}$.
The Chevalley group $D_4(q^3)$ has a field automorphism
$
F_q \colon \mathbb{F}_{q^3}  \to  \mathbb{F}_{q^3}:t\mapsto t^q
$
sending $x_r(t)$ to $x_r(t^q)$,
and a graph automorphism $\rho$ sending $x_r(t)$ to $x_{\rho(r)}(t)$ $(r\in \Phi_{D_4})$.
Let $F:=\rho {F_q}={F_q} \rho$.
For a subgroup $X$ of $D_4(q^3)$,
set $X^F:=\{x\in X| F(x)=x\}$.
Then we get $D_4(q^3)^F={^3}D_4(q^3)$.
Set
$r^1:=\frac{1}{3}(r+\rho(r)+\rho^2(r))$
{ for all }$r\in \Phi_{D_4}$.
Let $\alpha:= \frac{1}{3}(r_1+r_3+r_4)$ and $\beta:= r_2$,
then $\Delta_{G_2}:=\{ \alpha, \beta \}$ ($\alpha$ is short)
is a fundamental system of type $G_2$,
and the root system of type $G_2$ is
\begin{align*}
\Phi_{G_2}=& \{r^1 \mid  r\in \Phi_{D_4} \}
 =\{\pm \alpha, \pm\beta, \pm(\alpha+\beta),\pm(2\alpha+\beta),\pm(3\alpha+\beta),\pm(3\alpha+2\beta)\}.
\end{align*}

We obtain the following elements
of $^3D_4(q^3)$.
Let $r\in \Phi_{D_4}$ and $t\in \mathbb{F}_{q^3}$, we set
\begin{align*}
x_{r^1}(t):=
 {\left\{
 \begin{array}{ll}
 x_r(t)& \text{if } \rho(r)=r,{\ } t^q=t\\
 x_r(t)x_{\rho(r)}(t^q)x_{\rho^2(r)}(t^{q^2})
 & \text{if } \rho(r)\neq r,{\ } t^{q^3}=t\\
 \end{array}
 \right.}.
\end{align*}
The root subgroups of $^3D_4(q^3)$ are
$ X_{r^1}:=\{x_{r^1}(t) \mid  t\in \mathbb{F}_{q^3}\}$
for all $r\in \Phi_{D_4}$.

We write
\begin{alignat*}{2}
x_1(t):=& x_{r_1^1}(t)
=x_{r_1}(t)x_{r_3}(t^q)x_{r_4}(t^{q^2})
                & &\quad \text{for all } t\in \mathbb{F}_{q^3},\\
x_2(t):=& x_{r_2^1}(t)
=x_{r_2}(t)
        & & \quad \text{for all } t\in \mathbb{F}_{q},\\
x_3(t):=& x_{r_5^1}(t)
=x_{r_5}(t)x_{r_6}(t^q)x_{r_7}(t^{q^2})
        & &\quad \text{for all } t\in \mathbb{F}_{q^3},\\
x_4(t):=& x_{r_8^1}(t)
=x_{r_8}(t)x_{r_{10}}(t^q)x_{r_{9}}(t^{q^2})
        & & \quad \text{for all } t\in \mathbb{F}_{q^3},\\
x_5(t):=& x_{r_{11}^1}(t)
=x_{r_{11}}(t)
        & & \quad \text{for all } t\in \mathbb{F}_{q},\\
x_6(t):=& x_{r_{12}^1}(t)
=x_{r_{12}}(t)
        & &  \quad \text{for all } t\in \mathbb{F}_{q}.
\end{alignat*}
Set
$ X_i:= \{x_i(t) \mid  t \in {\mathbb{F}_{q^3}}\}$ ($i=1,3,4$) and
$ X_i:= \{x_i(t) \mid  t \in {\mathbb{F}_{q}}\}$  ($i=2,5,6$).
Note that the root subgroups are denoted by $Y_i$ in \cite{Le}, and
$Y_1:=X_2$, $Y_2:=X_1$ and $Y_i:=X_i$ for $i=3,4,5,6$.
We obtain a Sylow $p$-subgroup $U$ of $^3D_4(q^3)$:
\begin{align*}
 U
:=&\left\{x_2(t_2)x_1(t_1)x_3(t_3)x_4(t_4)x_5(t_5)x_6(t_6)
{\,}\middle |{\,}
t_1, t_3, t_4 \in {\mathbb{F}_{q^3}},{\ } t_2, t_5, t_6 \in {\mathbb{F}_{q}}\right\},
\end{align*}
where the product can be taken in an arbitrary, but fixed, order.

Let $t_1,t_3,t_4 \in \mathbb{F}_{q^3}$, $t_2,t_5,t_6 \in \mathbb{F}_{q}$
and define the commutator
\begin{align*}
[x_i(t_i), x_j(t_j)]:=x_i(t_i)^{-1}x_j(t_j)^{-1}x_i(t_i)x_j(t_j).
\end{align*}
Then the non-trivial commutators of $U$ are determined as follows:
\begin{align*}
[x_1(t_1), x_2(t_2)]=& x_3(-t_2t_1)x_4(t_2t_1^{q+1})x_5(-t_2t_1^{q^2+q+1})x_6(2t_2^2t_1^{q^2+q+1}),\\
[x_1(t_1), x_3(t_3)]=& x_4(t_1t_3^q+t_1^qt_3)
                       x_5(-t_1^{q+1}t_3^{q^2}-t_1^{q^2+q}t_3-t_1^{q^2+1}t_3^{q})
                      x_6(-t_1t_3^{q^2+q}-t_1^qt_3^{q^2+1}-t_1^{q^2}t_3^{q+1}),\\
[x_1(t_1), x_4(t_4)]=& x_5(t_1t_4^q+t_1^qt_4^{q^2}+t_1^{q^2}t_4),\\
[x_3(t_3), x_4(t_4)]=& x_6(t_3t_4^q+t_3^qt_4^{q^2}+t_3^{q^2}t_4),\\
[x_2(t_2), x_5(t_5)]=& x_6(t_2t_5).
\end{align*}
From now on, we set
 \begin{align*}
    x(t_1, t_2, t_3, t_4, t_5, t_6)
 := x_2(t_2)x_1(t_1)x_3(t_3)x_4(t_4)x_5(t_5)x_6(t_6)
 \in U.
 \end{align*}


\section{Conjugacy classes}
\label{conjugacy classes}
In this section, let $p$ be a fixed odd prime, we determine the conjugacy classes of
the Sylow $p$-subgroup $U$ of ${^3}D_4(q^3)$.

\begin{Notation}
 Let $x,u\in U$, then the conjugate of $x$ by $u$ is defined by
 ${^u}x:=uxu^{-1}$,
and the conjugacy class of $x$ is the set
${^U}x:=\{vxv^{-1}  \mid v\in U\}$.
\end{Notation}

\begin{Lemma}\label{phi_0,3D4}
 Let
$
  \phi_0 \colon  \mathbb{F}_{q^3} \to  \mathbb{F}_q: t \mapsto t+t^q+t^{q^2},
$
then $\phi_0$ is an $\mathbb{F}_q$-epimorphism and $|\ker\phi_0|=q^2$.
\end{Lemma}
\begin{proof}
Let $t,s\in \mathbb{F}_{q^3}$ and $k\in \mathbb{F}_{q}$, then
$
 \phi_0(t+s)= (t+s)+(t+s)^q+(t+s)^{q^2}=(t+t^q+t^{q^2})+(s+s^q+s^{q^2})
            = \phi_0(t)+\phi_0(s)$,
 and $\phi_0(kt)= (kt)+(kt)^q+(kt)^{q^2}=k(t+t^q+t^{q^2})
           = k\phi_0(t)$.
Thus $\phi_0$ is an $\mathbb{F}_q$-homomorphism.
Observe that
$ \ker \phi_0=\{t\in\mathbb{F}_{q^3}  \mid  t+t^q+t^{q^2}=0\}$.
The degree of $t+t^q+t^{q^2}$ is $q^2$, so $|\ker \phi_0|\leq q^2$.
On the other hand,
$ |\mathrm{im}{\,\phi_0}|\leq q$ since $\mathrm{im}{\,\phi_0}\subseteq \mathbb{F}_{q}$.
Then
$|\ker \phi_0|=\frac{|\mathbb{F}_{q^3}|}{|\mathrm{im}{\,\phi_0}|}\geq q^2$
since
$ \mathbb{F}_{q^3}/{\ker\phi_0}\cong \mathrm{im}{\,\phi_0}$.
Thus $|\ker \phi_0|=q^2$ and $\mathrm{im}{\,\phi_0}=\mathbb{F}_{q}$.
Therefore, $\phi_0$ is an $\mathbb{F}_q$-epimorphism.
\end{proof}

\begin{Proposition}\label{eta,3D4}
There exists an element $\eta\in {\mathbb{F}_{q^3}}\backslash {\mathbb{F}_{q}}$
(i.e. $\eta \in {\mathbb{F}_{q^3}}$ but $\eta \notin {\mathbb{F}_{q}}$)
such that
$\eta^{q^2}+\eta^{q}+\eta=1$.
\end{Proposition}
\begin{proof}
By \ref{phi_0,3D4}, we obtain
$\# \{t\in \mathbb{F}_{q^3} \mid t^{q^2}+t^{q}+t=1\}=q^2$.
But $|\mathbb{F}_{q}|=q$,
so there are at least $q^2-q$ $(>1)$ elements of $\mathbb{F}_{q^3}$
which satisfy the equation $t^{q^2}+t^{q}+t=1$.
Then the claim is proved.
\end{proof}

From now on, we fix an element $\eta\in {\mathbb{F}_{q^3}}\backslash {\mathbb{F}_{q}}$
such that $\eta^{q^2}+\eta^{q}+\eta=1$.
\begin{Corollary}\label{eta neq 0, 3D4}
 $1+\eta^{1-q^2}\neq 0$.
\end{Corollary}
\begin{proof}
 Suppose that $1+\eta^{1-q^2}= 0$.
 Multiply $\eta^{q^2}$ $(\neq 0)$ to both sides of the formula, then
$\eta^{q^2}+\eta= 0$, by
\ref{eta,3D4}
$\eta^{q}=1$,
so $\eta=1 \in \mathbb{F}_q$.
This is a contradiction to that  $\eta \in {\mathbb{F}_{q^3}}\backslash {\mathbb{F}_{q}}$.
\end{proof}

\begin{Notation/Lemma} \label{pi_q, 3D4}
The map
$\pi_q \colon  {\mathbb{F}_{q^3}}  \to  {\mathbb{F}_{q}}:
x\mapsto \phi_0(\eta x)=(\eta x)^{q^2}+(\eta x)^q+{\eta x}
$
is an $\mathbb{F}_q$-epimorphism and
$\mathbb{F}_{q^3}=\ker{\pi_q}\oplus \mathbb{F}_q$.
In particular, $\pi_q|_{\mathbb{F}_q}=\mathrm{id}_{\mathbb{F}_q}$ and $\pi_q^2=\pi_q$.
\end{Notation/Lemma}

\begin{Lemma}\label{zeta_u}
 Let $p>2$ and $u\in \mathbb{F}_{q^3}^*$, then the map
$
 \zeta_u \colon \mathbb{F}_{q^3} \to  \mathbb{F}_{q^3}: t\mapsto ut^{q^2}+u^qt^q
$
\nomenclature{$\zeta_u$}{an $\mathbb{F}_{q}$-automorphism for an fixed $u\in \mathbb{F}_{q^3}^*$:
              $\mathbb{F}_{q^3} \to  \mathbb{F}_{q^3}: t\mapsto ut^{q^2}+u^qt^q$
              \nomrefpage}%
is an $\mathbb{F}_{q}$-automorphism.
\end{Lemma}
\begin{proof}
Since taking the $q$-th power is an $\mathbb{F}_{q}$-homomorphism,
$\zeta_u$ is an $\mathbb{F}_{q}$-homomorphism.
Let $t\in \ker \zeta_u$, then $0=\zeta_u(t)=
ut^{q^2}+u^qt^q =(u^{1-q}t^{q^2-q}+1)u^qt^q=((u^{-1}t^q)^{q-1}+1)u^qt^q$.
Suppose that $t\neq 0$, then $(u^{-1}t^q)^{q-1}=-1$ and $(u^{-1}t^q)^{2(q-1)}=1$.
Thus the order of $(u^{-1}t^q)^{q-1}$ is $2$, i.e. $|(u^{-1}t^q)^{q-1}|=2$.
On the other hand, since $u^{-1}t^q \in \mathbb{F}_{q^3}^*$,
the order $|u^{-1}t^q|$ divides $(q^3-1)=(q-1)(q^2+q+1)$,
so $|(u^{-1}t^q)^{q-1}|$ divides $(q^2+q+1)$.
Since $q^2+q+1$ is odd, we get $|(u^{-1}t^q)^{q-1}|\neq 2$.
This is a contradiction, so $\ker \zeta_u=\{0\}$.
Thus
$\mathbb{F}_{q^3} \cong \mathrm{im}{\,\zeta_u}
  \subseteq \mathbb{F}_{q^3}$,
then $\mathrm{im}{\,\zeta_u} =\mathbb{F}_{q^3}$.
Therefore,
$\zeta_u$ is an $\mathbb{F}_{q}$-automorphism.
\end{proof}

\begin{Corollary}\label{zeta_u-b}
Let $p>2$,
the map
$
\mathbb{F}_{q^3} \to  \mathbb{F}_{q^3}: t\mapsto t+t^{q}
$
is an $\mathbb{F}_{q}$-automorphism.
\end{Corollary}
We obtain the following conjugate elements
from the commutator relations of $U$.
\begin{Lemma} \label{prop:conjugacy classes of x_i-3D4}
Let $u:=x(r_1,r_2,r_3,r_4,r_5,r_6)\in U$, $x_i(t_i)\in U$ and $\bar{t}_3\in \mathbb{F}_{q^3}$, then
\begin{align*}
{^u}x_6(t_6)=&x_6(t_6),\\
{^u}x_5(t_5)=&x_5(t_5)\cdot x_6(r_2t_5),\\
{^u}x_4(t_4)=&x_4(t_4)\cdot x_5(\phi_0\big(r_1t_4^q\big))
                      \cdot x_6(\phi_0\big(r_1r_2t_4^q\big)+\phi_0\big(r_3t_4^q\big)),\\
{^u}x_3(t_3)=&x_3(t_3)\cdot x_4(r_1t_3^q+r_1^qt_3)
                      \cdot x_5(\phi_0\big(r_1^{q^2+q}t_3\big))\\
                      & \cdot x_6(\phi_0\big(r_1^{q^2+q}r_2t_3\big)
                               +\phi_0\big(-r_1t_3^{q^2+q}\big)
                               +\phi_0\big(-t_3r_4^q\big)),\\
{^u}x_2(t_2)=&x_2(t_2)\cdot x_3(-r_1t_2)
                      \cdot x_4(-t_2r_1^{q+1})
                      \cdot x_5(-t_2r_1^{q^2+q+1})\\
                      & \cdot x_6(-t_2r_5-t_2^2r_1^{q^2+q+1}-t_2r_1^{q^2+q+1}r_2),\\
{^u}x_1(t_1)=&x_1(t_1)\cdot x_3(r_2t_1)
                      \cdot x_4(-r_2t_1^{q+1}-t_1r_3^q-t_1^qr_3)\\
                      & \cdot x_5(r_2t_1^{q^2+q+1}
                                 +\phi_0\big(r_1^{q^2}(-r_3t_1^q-r_3^qt_1) \big)
                                 +\phi_0\big(r_3^{q^2}t_1^{q+1} \big)
                                 +\phi_0\big(-t_1r_4^{q} \big)
                                 )\\
                      & \cdot x_6(2r_2^2t_1^{q^2+q+1}
                                  +\phi_0\big(r_1^{q^2}r_2(-r_3t_1^q-r_3^qt_1) \big)
                                  +\phi_0\big(r_2r_3^{q^2}t_1^{q+1} \big)\\
                      &\phantom{\cdot x_6(}
                                  +\phi_0\big(-r_2r_4^qt_1 \big)
                                  +\phi_0\big(-t_1r_3^{q^2+q} \big)),
\end{align*}
and
\begin{align*}
{^u}\big(x_3(t_3)x_5(t_5)\big)
                     =&x_3(t_3)\cdot x_4(r_1t_3^q+r_1^qt_3)
                      \cdot x_5(t_5+\phi_0\big(r_1^{q^2+q}t_3\big))\\
                      & \cdot x_6(r_2t_5+\phi_0\big(r_1^{q^2+q}r_2t_3\big)
                               +\phi_0\big(-r_1t_3^{q^2+q}\big)
                               +\phi_0\big(-t_3r_4^q\big)),
\end{align*}
\begin{align*}
{^u}\big(x_2(t_2)x_4(t_4)x_5(t_5)\big)
                      = &x_2(t_2)\cdot x_3(-r_1t_2)
                       \cdot x_4(t_4-t_2r_1^{q+1})
                       \cdot x_5(t_5-t_2r_1^{q^2+q+1}+\phi_0\big(r_1t_4^q\big))\\
                      & \cdot x_6(-t_2r_5-t_2^2r_1^{q^2+q+1}-t_2r_1^{q^2+q+1}r_2
                          +\phi_0\big(r_1r_2t_4^q\big)
                          +\phi_0\big(r_3t_4^q\big)
                          +r_2t_5),
\end{align*}
\begin{align*}
& {^u}\big(x_1(t_1)x_3(\bar{t}_3)\big)\\
       =&x_1(t_1)\cdot x_3(r_2t_1+\bar{t}_3)
                       \cdot x_4(-r_2t_1^{q+1}-t_1r_3^q-t_1^qr_3+r_1\bar{t}_3^q+r_1^q\bar{t}_3)\\
                      & \cdot x_5(r_2t_1^{q^2+q+1}
                                 +\phi_0\big(r_1^{q^2}(-r_3t_1^q-r_3^qt_1) \big)
                                 +\phi_0\big(r_3^{q^2}t_1^{q+1} \big)
                                 +\phi_0\big(-t_1r_4^{q} \big)
                                 +\phi_0\big(r_1^{q^2+q}\bar{t}_3\big)
                                 )\\
                      & \cdot x_6(2r_2^2t_1^{q^2+q+1}
                                  +\phi_0\big(r_1^{q^2}r_2(-r_3t_1^q-r_3^qt_1) \big)
                                  +\phi_0\big(r_2r_3^{q^2}t_1^{q+1} \big)
                                  +\phi_0\big(-r_2r_4^qt_1 \big)
                                  +\phi_0\big(-t_1r_3^{q^2+q} \big)
                               \\
                      &\phantom{\cdot x_6(}
                               +\phi_0\big(r_1^{q^2+q}r_2\bar{t}_3\big)
                               +\phi_0\big(-r_1\bar{t}_3^{q^2+q}\big)
                               +\phi_0\big(-\bar{t}_3r_4^q\big)
                               +\phi_0\big(\bar{t}_3^{q^2}(r_2t_1^{q+1}+t_1r_3^q+t_1^qr_3)\big)),
\end{align*}
\begin{align*}
  & {^u}\big(x_2(t_2)x_1(t_1)\big)\\
 =&x_2(t_2)x_1(t_1)\cdot x_3(r_2t_1-r_1t_2)
                      \cdot x_4(-r_2t_1^{q+1}-t_1r_3^q-t_1^qr_3
                                 -t_2r_1^{q+1}+t_1t_2r_1^q+t_1^qt_2r_1)\\
                      & \cdot x_5(r_2t_1^{q^2+q+1}
                                 +\phi_0\big(r_1^{q^2}(-r_3t_1^q-r_3^qt_1) \big)
                                 +\phi_0\big(r_3^{q^2}t_1^{q+1} \big)
                                 +\phi_0\big(-t_1r_4^{q} \big)
                                 \\
                      &\phantom{\cdot x_5(}
                                -t_2r_1^{q^2+q+1}
                                +\phi_0\big(-r_1t_1^{q^2+q}t_2\big)
                                +\phi_0\big(t_1^{q^2}t_2r_1^{q+1}\big) )\\
                      & \cdot x_6(2r_2^2t_1^{q^2+q+1}
                                  +\phi_0\big(r_1^{q^2}r_2(-r_3t_1^q-r_3^qt_1) \big)
                                  +\phi_0\big(r_2r_3^{q^2}t_1^{q+1}\big)
                                  +\phi_0\big(-r_2r_4^qt_1 \big)
                                  +\phi_0\big(-t_1r_3^{q^2+q} \big)
                                  \\
                      &\phantom{\cdot x_6(}
                        -t_2r_5-t_2^2r_1^{q^2+q+1}
                              -t_2r_1^{q^2+q+1}r_2
                              +\phi_0\big(-2r_1r_2t_1^{q^2+q}t_2\big)
                              +\phi_0\big(t_1^{q^2}t_2r_1^{q+1}r_2\big)\\
                      &\phantom{\cdot x_6(}
                              +\phi_0\big(r_1^{q^2+q}t_1t_2^2\big) ).
\end{align*}
\end{Lemma}

Before we get the conjugacy classes of $U$, we need the following lemma.
\begin{Lemma}\label{restriction of phi_0,3D4}
Let $u\in \mathbb{F}_{q^3}\backslash \mathbb{F}_q$, $\ker{\phi_0}$ be the kernel of $\phi_0$,
$S:=u\ker{\phi_0}$ and
\begin{align*}
  \phi_0|_{S} \colon u\ker{\phi_0}  \to  \mathbb{F}_{q}:t\mapsto t^{q^2}+t^{q}+t.
\end{align*}
Then $\phi_0|_{S}$ is an $\mathbb{F}_q$-epimorphism.
\end{Lemma}
\begin{proof}
The kernel $\ker{\phi_0}$ is an $\mathbb{F}_q$-vector subspace of $\mathbb{F}_{q^3}$
and $u\in \mathbb{F}_{q^3}$,
so $u\ker{\phi_0}$ is an $\mathbb{F}_q$-vector subspace of $\mathbb{F}_{q^3}$.
Thus $\phi_0|_{S}$ is an $\mathbb{F}_q$-homomorphism by \ref{phi_0,3D4}.
Let $t_0\in \ker{\phi_0}$ and $ut_0\in \ker{(\phi_0|_{S})}$, then
$t_0^{q^2}+t_0^{q}+t_0=0$ and $u^{q^2}t_0^{q^2}+u^{q}t_0^{q}+ut_0=0$,
so $u^{q^2}t_0^{q^2}+u^{q^2}t_0^{q}+u^{q^2}t_0=0$ since $u^{q^2}\neq 0$.
Then $(u^{q}-u^{q^2})t_0^{q}+(u-u^{q^2})t_0=0$.
Since $u\notin \mathbb{F}_q$, $u^{q}-u^{q^2}\neq 0$.
Then $(u^{q}-u^{q^2})^{-1}t_0=(u^{q}-u^{q^2})^{-q}t_0^q$,
so $(u^{q}-u^{q^2})^{-1}t_0 \in \mathbb{F}_q$ and $t_0 \in (u^{q}-u^{q^2})\mathbb{F}_q$.
Thus $|\ker{(\phi_0|_{S})}|\leq |u(u^{q}-u^{q^2})\mathbb{F}_q| =q$,
then $|\mathrm{im}{(\phi_0|_{S})}|\geq \frac{|S|}{|\ker{(\phi_0|_{S})}|}
              \stackrel{\ref{phi_0,3D4}}{=}\frac{q^2}{q}=q$.
We have $|\mathrm{im}{(\phi_0|_{S})}|=q$ since $|\mathrm{im}{(\phi_0|_{S})}|\leq |\mathbb{F}_q|=q$.
Therefore,  $\phi_0|_{S}$ is an $\mathbb{F}_q$-epimorphism.
\end{proof}

\begin{Notation}\label{notation:Ta}
Let $a^*\in \mathbb{F}_{q^3}^*=\mathbb{F}_{q^3}\backslash \{0\}$,
then denote by
 $T^{a^*}$ a complete set of
 coset representatives (i.e. a transversal)
 of  $(a^*\mathbb{F}_{q}^+)$ in $\mathbb{F}_{q^3}^+$.
\nomenclature{$T^{a^*}$}
{denotes a complete set of
 coset representatives
 of  $(a^*\mathbb{F}_{q}^+)$ in $\mathbb{F}_{q^3}^+$
 for an fixed $a^*\in \mathbb{F}_{q^3}^*$  \nomrefpage}%
Thus $|T^{a^*}|=q^2$.
If $\bar{t}_0\in T^{a^*}$ and $\bar{t}_0\in a^*\mathbb{F}_{q}^+$,
we set $\bar{t}_0=0$.
\end{Notation}

\begin{Proposition}[Conjugacy classes of $U$]\label{prop:conjugacy classes-3D4}
Let $p>2$, the conjugacy classes of the Sylow $p$-subgroup $U$ of ${^3}D_4(q^3)$ are
listed in Table \ref{table:conjugacy classes-3D4}.
\begin{table}[!htb]
\caption{Conjugacy classes of $U$ for $p>2$}
\label{table:conjugacy classes-3D4}
\begin{align*}
\begin{array}{|c|c|c|}\hline
 \rule{0pt}{13pt}
\begin{array}{c}
\text{Representatives } x \in U \\
\end{array}
& \text{Conjugacy Classes } {^Ux}
& |{^Ux}|
\\\hline
\hline
I_8
& x(0,0,0,0,0,0) & 1\\\hline
\begin{array}{c}
x_6(t_6^*), {\ }
t_6^*\in \mathbb{F}_q^*
\end{array}
& x(0,0,0,0,0,t_6^*) & 1\\\hline
\begin{array}{c}
x_5(t_5^*), {\ }
t_5^*\in \mathbb{F}_q^*
\end{array}
& \begin{array}{c}
x(0,0,0,0,t_5^*,s_6), \\
s_6\in \mathbb{F}_q
\end{array}
& q\\
\hline
\begin{array}{c}
x_4(t_4^*), {\ }
t_4^*\in \mathbb{F}_{q^3}^*
\end{array}
& \begin{array}{c}
x(0,0,0,t_4^*,s_5,s_6), \\
s_5,s_6\in \mathbb{F}_q
\end{array}
& q^2\\
\hline
\begin{array}{c}
x(0,0,t_3^*,0,t_5,0), \\
t_3^*\in \mathbb{F}_{q^3}^*,{\ }
t_5\in \mathbb{F}_q
\end{array}
& \begin{array}{c}
x(0,0,t_3^*,s_4,\hat{s}_5,s_6),\\
s_4\in \mathbb{F}_{q^3}, {\ } s_6\in \mathbb{F}_q
\end{array}
& q^4\\
\hline
\begin{array}{c}
x(0, t_2^*, 0, t_4, t_5, 0), \\
t_2^*\in \mathbb{F}_{q}^*,{\ }
t_4,t_5\in \mathbb{F}_q
\end{array}
&
\begin{array}{c}
x(0,t_2^*,s_3,\hat{s}_4,\hat{s}_5,s_6), \\
s_3\in \mathbb{F}_{q^3},{\ } s_6\in \mathbb{F}_q
\end{array}
&  q^4 \\
\hline
\begin{array}{c}
x(t_1^*,0,0,0,0,t_6) , \\
t_1^*\in \mathbb{F}_{q^3}^*,{\ }
t_6 \in \mathbb{F}_q
\end{array}
&
\begin{array}{c}
x(t_1^*,0,t_1^*s_2,s_4,s_5,\hat{s}_6),\\
s_4\in \mathbb{F}_{q^3},{\ } s_2,s_5\in \mathbb{F}_q
\end{array}
& q^5 \\\hline
\begin{array}{c}
x(t_1^*,0,{\bar{t}}_3^*,0,0,0), \\
t_1^*\in \mathbb{F}_{q^3}^*,{\ }
0\neq {\bar{t}}_3^*\in T^{t_1^*}
\end{array}
&
\begin{array}{c}
x(t_1^*,0,\bar{t}_3^*+t_1^*s_2,s_4,s_5,{s}_6),\\
s_4\in \mathbb{F}_{q^3},{\ } s_2,s_5,s_6\in \mathbb{F}_q
\end{array}
& q^6 \\\hline
\begin{array}{c}
x(t_1^*, t_2^*,0,0,0,0), \\
t_1^*\in \mathbb{F}_{q^3}^*,{\ }
t_2^* \in \mathbb{F}_q^*
\end{array}
&\begin{array}{c}
x(t_1^*,t_2^*,s_3,s_4,s_5,s_6),\\
s_3,s_4\in \mathbb{F}_{q^3},{\ } s_5,s_6\in \mathbb{F}_q
\end{array}
& q^8\\\hline
\end{array}
\end{align*}
where
$\hat{s}_{-}$ is determined by some of $t_{-}^*$, $t_{-}$ and $s_{-}$.
The entries in the 1st column are the representatives
of the distinct conjugacy classes of $U$.
For a fixed representative $x\in U$,
$^Ux$ in the 2nd column is the conjugacy class  of $x$,
and $|^Ux|$  in the 3rd column is the number of
the elements of $^Ux$.
\end{table}
\end{Proposition}
\begin{proof}
We prove the hard cases of the proposition.
Let
$0\neq t_1\in \mathbb{F}_{q^3}^*$, $\bar{t}_3\in \mathbb{F}_{q^3}$, $t_6\in  \mathbb{F}_{q}$,
$u:=x(r_1,r_2,r_3,r_4,r_5,r_6)\in U$
 and
$x(a_1,a_2,a_3,a_4,a_5,a_6):=
   {^u}\big(x_1(t_1)x_3(\bar{t}_3)x_6(t_6)\big)$.

\begin{itemize}
 \item [(1)] Let $\bar{t}_3\in t_1\mathbb{F}_q^+$,
 then there exists $s\in \mathbb{F}_q^+$ such that $\bar{t}_3=st_1$.
 By \ref{prop:conjugacy classes of x_i-3D4},
 \begin{align*}
  a_1=& t_1,{\ }
  a_2= 0,{\ }
  a_3= r_2t_1+\bar{t}_3=(r_2+s)t_1,{\ }
  a_4= -r_2t_1^{q+1}-(r_3-sr_1)t_1^q-(r_3-sr_1)^qt_1,\\
  a_5=& r_2t_1^{q^2+q+1}
       +\phi_0\big(r_1^{q^2}(-r_3t_1^q-r_3^qt_1) \big)
                                 +\phi_0\big(r_3^{q^2}t_1^{q+1} \big)
                                 +\phi_0\big(-t_1r_4^{q} \big)
                                 +\phi_0\big(sr_1^{q^2+q}{t}_1\big),\\
  a_6=& t_6+r_2^2t_1^{q^2+q+1}+r_2a_5+\phi_0\big(-t_1r_3^{q^2+q} \big)
                               +\phi_0\big(-s^2r_1{t}_1^{q^2+q}\big)\\
       & +\phi_0\big(-s{t}_1r_4^q\big)
                    +\phi_0\big(s{t}_1^{q^2}(r_2t_1^{q+1}+t_1r_3^q+t_1^qr_3)\big)\\
     =& t_6+r_2^2t_1^{q^2+q+1}+(r_2+s)a_5+2sr_2t_1^{q^2+q+1} \\
       &+\phi_0\big(s(r_3-sr_1)t_1^{q^2+q}\big)
        -\phi_0\big((r_3-sr_1)^{q^2+q}t_1\big).
 \end{align*}
Let $a_3$, $a_4$ and $a_5$ be fixed,
then $a_6$ is determined uniquely
by \ref{phi_0,3D4} and \ref{zeta_u}.
Hence we get the conjugacy class of $x_1(t_1)x_6(t_6)$.
\begin{align*}
 {^U}\big(x_1(t_1)x_6(t_6)\big)=
 \left\{x(t_1,0,s_2t_1,s_4,s_5,\hat{s}_6) \mid s_4\in \mathbb{F}_{q^3}, s_2,s_5\in\mathbb{F}_{q} \right\}.
\end{align*}
 \item [(2)] Let $\bar{t}_3\notin t_1\mathbb{F}_q^+$.
 By \ref{prop:conjugacy classes of x_i-3D4},
 \begin{align*}
  a_1=& t_1,\quad
  a_2= 0,\quad
  a_3= \bar{t}_3+r_2t_1,\quad
  a_4= -r_2t_1^{q+1}-t_1r_3^q-t_1^qr_3+r_1\bar{t}_3^q+r_1^q\bar{t}_3,\\
  a_5=& r_2t_1^{q^2+q+1}
                                 +\phi_0\big(r_1^{q^2}(-r_3t_1^q-r_3^qt_1) \big)
                                 +\phi_0\big(r_3^{q^2}t_1^{q+1} \big)
                                 +\phi_0\big(-t_1r_4^{q} \big)
                                 +\phi_0\big(r_1^{q^2+q}\bar{t}_3\big),\\
  a_6=& t_6+r_2^2t_1^{q^2+q+1}+r_2a_5+\phi_0\big(-t_1r_3^{q^2+q} \big)
                               +\phi_0\big(-r_1\bar{t}_3^{q^2+q}\big)\\
       & +\phi_0\big(-\bar{t}_3r_4^q\big)
                    +\phi_0\big(\bar{t}_3^{q^2}(r_2t_1^{q+1}+t_1r_3^q+t_1^qr_3)\big).
 \end{align*}

 For the fixed $r_1,r_2,r_3$ and $a_5$,
 let
 \begin{align*}
 T:=\{t_1r_4^q \mid
 r_4\in \mathbb{F}_{q^3}, \text {and }
 r_1,r_2,r_3, a_5 {\text{ are fixed}}
 \}.
 \end{align*}
By \ref{phi_0,3D4}, $|T|=q^2$.
Let $x_0\in T$, then $T=x_0+\ker{\phi_0}$ and
$\bar{t}_3r_4^q\in \frac{\bar{t}_3x_0}{t_1}+\frac{\bar{t}_3}{t_1}\ker{\phi_0}$.
We know $\bar{t}_3\notin t_1\mathbb{F}_q^+$,
so $\frac{\bar{t}_3}{t_1}\in \mathbb{F}_{q^3}^+\backslash \mathbb{F}_q^+$.
Hence
\begin{align*}
  & \{\phi_0(\bar{t}_3r_4^q) \mid t_1r_4^q\in T\}
 = \left\{\phi_0(\frac{\bar{t}_3x_0}{t_1})+\phi_0(t)
    {\, }\middle|{\, }
           t\in \frac{\bar{t}_3}{t_1}\ker{\phi_0} \right\}
 \stackrel{\ref{restriction of phi_0,3D4}}{=} \mathbb{F}_q^+.
\end{align*}
Thus $a_6$ can be every element of $\mathbb{F}_q$ for the fixed $r_1,r_2,r_3$ and $a_5$.
Therefore,
\begin{align*}
 {^U}\big(x_1(t_1)x_3(\bar{t}_3)\big)=
 \left\{x(t_1,0,\bar{t}_3+s_2t_1,s_4,s_5,{s}_6) \mid
          s_4\in \mathbb{F}_{q^3}, s_2, s_5, s_6\in\mathbb{F}_{q} \right\}.
\end{align*}
\end{itemize}
By \ref{prop:conjugacy classes of x_i-3D4}, the other conjugacy classes are also determined.
Then the proposition is obtained.
\end{proof}

\begin{Remark}
\label{conjugacy classes p=2}
If $p=2$, the map $\zeta_u$ in Lemma \ref{zeta_u} is not surjective.
Similarly, Corollary \ref{zeta_u-b} does not hold either.
Thus there are some differences between the classification of conjugacy classes of $U$
for $p=2$ and for $p>2$.
Some results still hold when their proofs does not use the surjective property.
Let $x:=x(t_1, t_2, t_3, t_4, t_5, t_6)\in U$.
If $t_1=t_3=0$, the classifications of conjugacy classes of $x$ for $p=2$ and these for $p>2$
are the same.
Otherwise, the analysis of conjugacy classes of $x$ for $p=2$ and this for $p>2$
are different.
For example, let $t_3^* \in \mathbb{F}_{q^3}^*$,  then
$\{{t_3^*}r^q+{t_3^*}^qr\mid r\in \mathbb{F}_{q^3}\}=
\{\zeta_{t_3^*}(r^{q^2})\mid r\in \mathbb{F}_{q^3}\}
=\mathrm{im}{\,\zeta_{t_3^*}}$ and $|\mathrm{im}{\,\zeta_{t_3^*}}|=q^2$.
Similar to the proof of Proposition \ref{prop:conjugacy classes-3D4},
we obtain the conjugacy class of $x_3(t_3^*)$ for $p=2$ as follows:
\begin{align*}
{^Ux_3(t_3^*)}=&
 {\left\{
 \begin{array}{ll}
 \{x(0,0,t_3^*,\tilde{s}_4, s_5,s_6) \in U \mid
\tilde{s}_4\in \mathrm{im}{\,\zeta_{t_3^* }}, {\ } s_5, s_6\in \mathbb{F}_q \}
& \text{if } p=2\\
 \{x(0,0,t_3^*,s_4,\hat{s}_5,s_6) \in U \mid
s_4\in \mathbb{F}_{q^3}, {\ } s_6\in \mathbb{F}_q \}
& \text{if } p>2\\
 \end{array}
 \right.},
\end{align*}
where $\hat{s}_5\in \mathbb{F}_q$ is determined by $t_3^*$ and $s_4$ (see Proposition \ref{prop:conjugacy classes-3D4}).
\end{Remark}


\section{Irreducible characters}
\label{sec:irreducible characters}
In \cite{Le}, Tung Le constructed and counted all ordinary irreducible characters of $U$,
mainly using Clifford theory.
In this section,
we give the specific constructions of the irreducible characters of $U$
using Clifford's Theorem (see \cite{CR1})
when $p>2$,
so that we can determine the values of the irreducible characters in the next section.

Let $G$ be a finite group,
$N$ a normal subgroup of $G$,
and $K$ a field.
Let $\mathrm{Irr}(G)$ be the set of all complex irreducible characters of $G$,
and $\mathrm{triv}_{G}$  the trivial character of $G$.
Let $H$ be a subgroup of $G$,
$\chi\in \mathrm{Irr}(G)$ and $\lambda \in \mathrm{Irr}(H)$,
then
we denote by $\mathrm{Ind}_H^G{\lambda}$ the character induced from $\lambda$,
and deonte by $\mathrm{Res}^G_H{\chi}$ the restriction of $\chi$ to $H$.
The center of $G$ is denoted by $Z(G)$.
The kernel of $\chi$ is
              $\ker{\chi}:=\{g\in G \mid \chi(g)=\chi(1)\}$.
The commutator subgroup of $G$ is
             $G'=\left<{\,} [x,y] \mid  x,y\in G {\,}\right>$,
           where $[x,y]=x^{-1}y^{-1}xy$.
Let $\lambda \in \mathrm{Irr}(N)$,
then the inertia group in $G$ is
$I_G(\lambda)=\{g\in G \mid  \lambda^g=\lambda \}$
where $\lambda^g(n)=\lambda(gng^{-1})$ for all $n\in N$.
In paricular, $N\unlhd I_G(\lambda) \leqslant G$.

\begin{Notation/Lemma}\label{vartheta}
 Let $\vartheta \colon \mathbb{F}_q^+ \to  \mathbb{C}^*$ denote a fixed nontrivial
 linear character of the additive group $\mathbb{F}_q^+$
 of $\mathbb{F}_q$
 once and for all.
 In particular, $\sum_{x\in \mathbb{F}_q^+}{\vartheta(x)}=0$.
\end{Notation/Lemma}

We determine the irreducible characters of $\mathbb{F}_{q^3}^+$ and $\mathbb{F}_{q}^+$.
\begin{Lemma}\label{vartheta-a,b}
Let $a\in \mathbb{F}_{q^3}$, $b\in \mathbb{F}_{q}$ and
\begin{align*}
\vartheta_a \colon & \mathbb{F}_{q^3}^+ \to  \mathbb{C}^*:x\mapsto \vartheta\pi_q(ax),\\
\vartheta_b \colon & \mathbb{F}_q^+ \to  \mathbb{C}^*:y\mapsto \vartheta\pi_q(by)=\vartheta(by),
\end{align*}
then
$\mathrm{Irr}(\mathbb{F}_{q^3}^+)= \{\vartheta_a \mid  a\in \mathbb{F}_{q^3}\}$
and
$\mathrm{Irr}(\mathbb{F}_{q}^+)= \{\vartheta_b \mid  b\in \mathbb{F}_{q}\}$.
\end{Lemma}
\begin{proof}
Let $a\in \mathbb{F}_{q^3}$.
By \ref{pi_q, 3D4} and \ref{vartheta}, we get
$
 \vartheta_a(x+y)=\vartheta_a(x)\cdot \vartheta_a(y)$
 { for all } $x,y\in  \mathbb{F}_{q^3}$,
so
$\mathrm{Irr}(\mathbb{F}_{q^3}^+)\supseteq \{\vartheta_a \mid  a\in \mathbb{F}_{q^3}\}$.
Let $a, c\in \mathbb{F}_{q^3}$ and $\vartheta_a=\vartheta_c$, then
$\vartheta_a(x)=\vartheta_c(x)$ for all $x\in  \mathbb{F}_{q^3}$,
i.e. $\vartheta\pi_q(ax)=\vartheta\pi_q(cx)$ for all $x\in  \mathbb{F}_{q^3}$.
Thus $a=c$ by \ref{pi_q, 3D4}.
Otherwise, $\vartheta(y)=1$ for all $ y\in  \mathbb{F}_{q}$,
this is a contradiction since $\vartheta$ is nontrivial by \ref{vartheta}.
Thus $\#\{\vartheta_a \mid  a\in \mathbb{F}_{q^3}\}=q^3$.
Therefore, $\mathrm{Irr}(\mathbb{F}_{q^3}^+)= \{\vartheta_a \mid  a\in \mathbb{F}_{q^3}\}$.
Similarly, we obtain the second formula.
\end{proof}

\begin{Corollary}\label{vartheta-A15b}
Let $u\in \mathbb{F}_{q^3}$, and
$\tilde{\vartheta}_{u} \colon  \mathbb{F}_{q^3}^+ \to  \mathbb{C}^*
:t\mapsto \vartheta\pi_q(u^qt+ut^{q^2})$,
then
$\tilde{\vartheta}_{u}=\vartheta_{\eta^{-1}(\eta+\eta^{q})u^q}$
and
$ \mathrm{Irr}(\mathbb{F}_{q^3}^+)= \{\tilde{\vartheta}_u \mid  u\in \mathbb{F}_{q^3}\}$.
\end{Corollary}
\begin{proof}
For all $t\in \mathbb{F}_{q^3}^+$,
$
\tilde{\vartheta}_{u}(t)
= \vartheta\pi_q(u^qt+ut^{q^2})
\stackrel{\ref{pi_q, 3D4}}{=}  \vartheta\phi_0(\eta u^qt+\eta ut^{q^2})
\stackrel{\ref{phi_0,3D4}}{=}  \vartheta\phi_0((\eta+\eta^{q}) u^{q}t)\\
{=}  \vartheta\pi_q(\eta^{-1}(\eta+\eta^{q}) u^{q}t)
 =\vartheta_{\eta^{-1}(\eta +\eta^{q})u^{q}}(t)$.
 By \ref{eta neq 0, 3D4}, $\eta +\eta^q\neq 0$.
 By \ref{vartheta-a,b},
 we obtain the second formula.
\end{proof}

\begin{Lemma}\label{restriction center-3D4}
 Let $G$ be a finite group, $Z(G)\subseteq N\trianglelefteq G$,
 and $\chi \in \mathrm{Irr}(G)$.
 Let $\lambda\in \mathrm{Irr}(N)$
 such that $\langle \mathrm{Res}^G_N\chi,  \lambda\rangle_N=e >0$.
 Then for all $g\in Z(G)$,
 \begin{align*}
  \left(\mathrm{Res}^G_N\chi \right)(g)=e \frac{|G|}{|I_G(\lambda)|}\lambda(g)
 \end{align*}
and $g\notin \ker \chi \iff  g \notin \ker \lambda$.
In particular,  if $X \leqslant Z(G)$, then  $X \nsubseteq \ker\chi$
if and only if $X \nsubseteq \ker\lambda$.
\end{Lemma}
\begin{proof}
By Clifford's Theorem,
we have for all $g\in Z(G)$
 \begin{align*}
\left(\mathrm{Res}^G_N\right)\chi(g)
=e \sum_{h\in T_{I_G(\lambda)\backslash G}} {\lambda^h}(g)
\stackrel{g\in Z(G)}{=}e \frac{|G|}{|I_G(\lambda)|}\lambda(g).
 \end{align*}
Let $1$ be the identity element of $G$
and  $g_0\in Z(G)$,
then
\begin{align*}
     \chi(g_0) \neq &  \chi(1)
           \iff     \left(\mathrm{Res}^G_N\chi\right)(g_0) \neq  \left(\mathrm{Res}^G_N\chi\right)(1)
           \iff    \lambda(g_0) \neq  \lambda(1).
\end{align*}
Thus $g\notin \ker \chi \iff  g \notin \ker \lambda$.
\end{proof}

\begin{Definition}[2.1, \cite{Le}]
 Let $\chi$ be an irreducible character of a group $G$.
 $\chi$ is said to be \textbf{almost faithful} if $Z(G)\nleqslant \ker\chi$.
\end{Definition}

We determine some inertia groups.
\begin{Lemma}\label{some inertia groups N-3D4}
Let $T:=X_2X_3X_4X_5X_6$,
$N:=X_4X_5X_6$
and $H:=X_1X_4X_5X_6$.
 \begin{itemize}
 \item[(1)]
The subgroup $N$ is abelian,
$N\trianglelefteq U$, $T\trianglelefteq U$ and $H\leqslant U$ as follows:
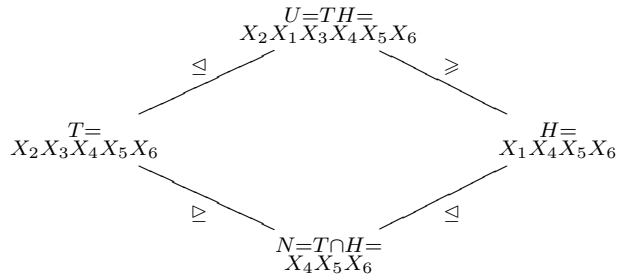
\begin{align*}
 \xymatrix{
               &  {\substack{U=TH=\\X_2X_1X_3X_4X_5X_6}} \ar@{-}[dl]_{\trianglelefteq} \ar@{-}[dr]^{\geqslant}
               &      \\
{\substack{ T=\\X_2X_3X_4X_5X_6}} \ar@{-}[dr]_{\trianglerighteq}
 &
 & {\substack{ H=\\X_1X_4X_5X_6}} \ar@{-}[dl]^{\trianglelefteq}  \\
               & {\substack{N=T \cap H= \\X_4X_5X_6}}       &}
\end{align*}
\item[(2)] Let $\lambda \in \mathrm{Irr}(N)$ and $\mathrm{Res}^N_{X_6}\lambda \neq \mathrm{triv}_{X_6}$.
If $\lambda$ satisfies that $\mathrm{Res}^N_{X_5}\lambda = \mathrm{triv}_{X_5}$, we have
\begin{align*}
 I_{U}(\lambda)=\{u\in U \mid \lambda^u=\lambda\}=H.
\end{align*}
\item[(3)]
Let $\lambda \in \mathrm{Irr}(N)$,
the inertia group $I_{T}(\lambda)$ is
\begin{align*}
 I_{T}(\lambda)=
 {\left\{
 \begin{array}{ll}
  T & \text{if }\mathrm{Res}^N_{X_6}{\lambda} =\mathrm{triv}_{X_6}\\
  N & \text{if }\mathrm{Res}^N_{X_6}{\lambda} \neq \mathrm{triv}_{X_6}
 \end{array}
 \right.}.
\end{align*}
\item[(4)]  Let $\lambda \in \mathrm{Irr}(N)$,
then the inertia group $I_{H}(\lambda)$ is
\begin{align*}
 I_{H}(\lambda)=
 {\left\{
 \begin{array}{ll}
  H & \text{if }\mathrm{Res}^N_{X_5}{\lambda} =\mathrm{triv}_{X_5}\\
  N & \text{if }\mathrm{Res}^N_{X_5}{\lambda} \neq \mathrm{triv}_{X_5}
 \end{array}
 \right.}.
\end{align*}
\item [(5)]
Let $\psi \in \mathrm{Irr}(T)$ and $X_6=Z(T)\nsubseteq \ker\psi$,
then the inertia group $I_{U}(\psi)$ is
\begin{align*}
 I_{U}(\psi)=\{u\in U \mid \psi^u=\psi\}=U.
\end{align*}
\end{itemize}
\end{Lemma}
\begin{proof}
\begin{itemize}
 \item [(1)]
 From the commutator relations of $U$  and by \ref{prop:conjugacy classes of x_i-3D4},
 we obtain that $N$ is abelian,
 $N$ and $T$ are normal subgroups of $U$, and $H$ is a subgroup of $U$.
 Then $N$ is a normal subgroup of $N$ and $H$, since $N$ is a normal subgroup of $U$.
  \item [(2)]
  Let
  $u:=x(r_1,r_2,r_3,r_4,r_5,r_6)\in U$,  $x:=x_4(t_4)x_5(t_5)x_6(t_6)\in N$,
  and $\lambda \in \mathrm{Irr}(N)$,
  then $\lambda$ is a linear character since $N$ is abelian.
  We have
  \begin{align*}
   \lambda^u(x)
   =& \lambda(u\cdot x\cdot u^{-1})\\
   \stackrel{\ref{prop:conjugacy classes of x_i-3D4}}{=}
   & \lambda\left(x_4(t_4)x_5(t_5+\phi_0(r_1t_4^q))
             x_6(t_6+r_2t_5+\phi_0(r_1r_2t_4^q)+\phi_0(r_3t_4^q))\right)\\
   \stackrel{\mathrm{Res}^N_{X_5}\lambda = \mathrm{triv}_{X_5}}{=}
   & \lambda(x_4(t_4))
      \cdot \lambda(x_6(t_6+r_2t_5+\phi_0(r_1r_2t_4^q)+\phi_0(r_3t_4^q))).
  \end{align*}
Since $X_6\nsubseteq \ker\lambda$ and $X_5\subseteq \ker\lambda$, we get
  \begin{align*}
  I_{U}(\lambda)
 & = \{u\in U  \mid  \lambda^u=\lambda \}
 = \{u\in U  \mid  \lambda(u\cdot x\cdot u^{-1})=\lambda(x)
                      \text{ for all }x \in N\} \\
& \stackrel{\ref{phi_0,3D4}}{=}
          \{u:=x(r_1,r_2,r_3,r_4,r_5,r_6)\in U
          \mid r_2=r_3=0 \}
= X_1X_4X_5X_6
= H.
\end{align*}
\item [(3)] c.f. the proof of (2).
\item [(4)] c.f. the proof of (2).
\item [(5)]
Let $x_6(t_6)\in X_6$.
Since $X_6=Z(T)\trianglelefteq N\trianglelefteq T$,
there exists $\lambda \in \mathrm{Irr}(N)$ such that
$0<e=\langle \mathrm{Res}^T_N\psi, \lambda\rangle_N$,
then
$
\left(\mathrm{Res}^T_N\psi\right)(x_6(t_6)) =e \frac{|T|}{|I_T(\lambda)|}\lambda(x_6(t_6))$.
By \ref{restriction center-3D4},
$X_6=Z(T)\nsubseteq \ker\lambda$
since  $X_6=Z(T)\nsubseteq \ker\psi$.
Then $\mathrm{Res}^N_{X_6}\lambda \neq \mathrm{triv}_{X_6}$.
Then $I_T(\lambda)=N$ by (3),
so $\mathrm{Ind}_N^T\lambda \in \mathrm{Irr}(T)$ by Clifford's Theorem.
Thus $e=1$ and $\psi=\mathrm{Ind}_N^T\lambda$.

Let $u:=x(s_1,s_2,s_3,s_4,s_5,s_6)\in U$ and $h:=x(0,t_2,t_3,t_4,t_5,t_6)\in T$, then
\begin{align*}
 \psi(h)=\left(\mathrm{Ind}_N^T\lambda\right)(h)
= \frac{1}{|N|}\sum_{\substack{y\in T\\ y\cdot h\cdot y^{-1}\in N}}\lambda(y\cdot h\cdot y^{-1}).
\end{align*}
\begin{itemize}
 \item  Let $t_2\neq 0$ or $t_3\neq 0$, then $\psi(h)=\left(\mathrm{Ind}_N^T\lambda\right)(h)=0$ and
       \begin{align*}
  (\psi^u)(h)=\left(\mathrm{Ind}_N^T\lambda\right)^u(h)
  =\left(\mathrm{Ind}_N^T\lambda\right)(u\cdot h\cdot u^{-1})
  \stackrel{\ref{prop:conjugacy classes of x_i-3D4}}{=}0.
       \end{align*}
 \item  Let $t_2=t_3= 0$, then by \ref{prop:conjugacy classes of x_i-3D4}
\begin{align*}
 \psi(h)=\left(\mathrm{Ind}_N^T\lambda\right)(h)
=& \frac{1}{|N|}\sum_{\substack{y\in T\\ y\cdot h\cdot y^{-1}\in N}}\lambda(y\cdot h\cdot y^{-1})
           {=} \frac{1}{|N|}\sum_{y\in T}\lambda(y\cdot h\cdot y^{-1})\\
\stackrel{y:=x(0,r_2,r_3,r_4,r_5,r_6)}{=}&
         \frac{1}{|N|}\sum_{\substack{r_2,r_5,r_6\in \mathbb{F}_q \\r_3,r_4\in \mathbb{F}_{q^3}}}
         \lambda(x_4(t_4)x_5(t_5)x_6(t_6+r_2t_5+\phi_0(r_3t_4^q))),
\end{align*}
and
\begin{align*}
&(\psi^u)(h)
= \left(\mathrm{Ind}_N^T\lambda\right)^u(h)
= \left(\mathrm{Ind}_N^T\lambda\right)(u\cdot h\cdot u^{-1})\\
=& \left(\mathrm{Ind}_N^T\lambda\right)(x_4(t_4)x_5(t_5+\phi_0(s_1t_4^q))
                           x_6(t_6+s_2t_5+\phi_0(s_1s_2t_4^q)+\phi_0(s_3t_4^q)))\\
{=}&\frac{1}{|N|}\sum_{\substack{r_2,r_5,r_6\in \mathbb{F}_q \\r_3,r_4\in \mathbb{F}_{q^3}}}
         \lambda(x_4(t_4)x_5(t_5+\phi_0(s_1t_4^q))x_6(t_6+s_2t_5+\phi_0(s_1s_2t_4^q)+\phi_0(s_3t_4^q))\\
   & \phantom{\frac{1}{|N|}\sum_{\substack{r_2,r_5,r_6\in \mathbb{F}_q \\r_3,r_4\in \mathbb{F}_{q^3}}}}
          \cdot x_6(r_2t_5+r_2\phi_0(s_1t_4^q)+\phi_0(r_3t_4^q))).
\end{align*}
\end{itemize}
Since $X_6\nsubseteq \ker\lambda$, we have for all $u\in U$
\begin{align*}
 (\psi^u)(h)=\psi(h)
 =&\left\{
 \begin{array}{ll}
  q^4 \lambda(x_6(t_6)) & \text{ if } t_2=t_3=t_4=t_5= 0\\
  0  & \text{ otherwise }
 \end{array}
 \right..
\end{align*}
\end{itemize}
Thus $I_U(\psi)=U$.
\end{proof}

We determine the irreducible characters of the abelian group $X_4X_5X_6$.
From now on, we set $N:=X_4X_5X_6$.
\begin{Lemma}\label{irr. char. X456-3D4}
 Let  $A_{17}, A_{16}\in \mathbb{F}_q$, $A_{15}\in \mathbb{F}_{q^3}$ and
     \begin{align*}
       \lambda^{A_{17}, A_{16}, A_{15}}(x_4(t_4)x_5(t_5)x_6(t_6))
          :=\vartheta(A_{17}t_6)
          \cdot \vartheta(A_{16}t_5)
          \cdot \vartheta\pi_q(A_{15}^qt_4+A_{15}t_4^{q^2}),
      \end{align*}
then
$\mathrm{Irr}(N)
 = \left\{ \lambda^{A_{17}, A_{16}, A_{15}}
      {\,}\middle|{\,} A_{17}, A_{16}\in \mathbb{F}_q, A_{15}\in \mathbb{F}_{q^3} \right\}$.
\end{Lemma}

Now we determine the irreducible characters of the subgroup
$H=X_1X_4X_5X_6$ of $U$.
\begin{Lemma}\label{irr. char. X1456-3D4}
Let $H:=X_1X_4X_5X_6$ and $\tilde{\chi}\in \mathrm{Irr}(H)$,
then $H'=X_5$.
\begin{itemize}
 \item [(1)] If $X_5\subseteq \ker\tilde{\chi}$,
 let
       $\bar{H}_{146}:={{X_5}\backslash H} \cong \bar{X}_1\bar{X}_4\bar{X}_6$,
       $ \bar{\chi}^{A_{17}, A_{15}, A_{12}}\in \mathrm{Irr}(\bar{H}_{146})$,
      \begin{align*}
       & \bar{\chi}^{A_{17}, A_{15}, A_{12}}(\bar{x}_1(t_1)\bar{x}_4(t_4)\bar{x}_6(t_6))
          :=\vartheta(A_{17}t_6)
          \cdot \vartheta\pi_q(A_{15}^qt_4+A_{15}t_4^{q^2})
          \cdot \vartheta\pi_q(A_{12}t_1).
      \end{align*}
 Denote by $\tilde{\chi}^{A_{17}, A_{15}, A_{12}}$ the lift of
 $\bar{\chi}^{A_{17}, A_{15}, A_{12}}$ to $H$, then
 \begin{align*}
  \mathrm{Irr}(H)_1:=& \{\tilde{\chi}\in \mathrm{Irr}(H)  \mid  X_5\subseteq \ker\tilde{\chi}\}
 = \left\{\tilde{\chi}^{A_{17}, A_{15}, A_{12}}\in \mathrm{Irr}(H)
      {\,}\middle|{\,} A_{17}\in \mathbb{F}_q, A_{15},A_{12}\in \mathbb{F}_{q^3} \right\}.
 \end{align*}
 \item [(2)] If $X_5\nsubseteq \ker\tilde{\chi}$,
we have
 \begin{align*}
 \mathrm{Irr}(H)_2:= & \left\{\tilde{\chi}\in \mathrm{Irr}(H) {\,}\middle|{\,} X_5\nsubseteq \ker\tilde{\chi}\right\}
 = \left\{\mathrm{Ind}_N^H\lambda^{A_{17}, A_{16}^*, 0}
      {\,}\middle|{\,} A_{17}\in \mathbb{F}_q, A_{16}^*\in \mathbb{F}_q^* \right\}.
 \end{align*}
\end{itemize}
Thus, $\mathrm{Irr}(H)=\mathrm{Irr}(H)_1 \dot{\cup} \mathrm{Irr}(H)_2$,
i.e. $H$ has $q^7$ linear characters
and $(q-1)q$ irreducible characters of degree $q^3$.
\end{Lemma}
\begin{proof}
Let $\tilde{\chi}\in \mathrm{Irr}(H)$.
From the commutator relations, $H'=X_5$ and $Z(H)=X_5X_6$.
\begin{itemize}
 \item [(1)] If $X_5\subseteq \ker\tilde{\chi}$,
 $\tilde{\chi}$ is linear.
Since $H'=X_5$,
all linear characters of $H$
are precisely the lifts to $H$ of the irreducible characters of
the abelian quotient group $X_5\backslash H$.
 \item [(2)] If $X_5\nsubseteq \ker\tilde{\chi}$,
let
$\lambda\in \mathrm{Irr}(N)$
    and $\langle\mathrm{Res}^H_N{\tilde{\chi}}, \lambda \rangle_N >0$.
Since $X_5 \trianglelefteq Z(H) \trianglelefteq N\trianglelefteq H$,
    by \ref{restriction center-3D4} $X_5\nsubseteq \ker\lambda$
    i.e. $\mathrm{Res}^N_{X_5}\lambda \neq \mathrm{triv}_{X_5}$,
then $I_H(\lambda)=N$ by (4) of \ref{some inertia groups N-3D4}.
By Clifford's Theorem,
   $\tilde{\chi}=\mathrm{Ind}_N^H \lambda$ and $\deg(\tilde{\chi})=q^3$.
Then there exists $\lambda^{A_{17}, A_{16}^*, A_{15}}\in \mathrm{Irr}(N) $
   such that $\lambda=\lambda^{A_{17}, A_{16}^*, A_{15}}$.
We note that $X_1$ is a transversal of $N$ in $H$.
Let $r_1\in \mathbb{F}_{q^3}$,
for all $x_4(t_4)x_5(t_5)x_6(t_6)\in N$,
\begin{align*}
  &\left(\lambda^{A_{17}, A_{16}^*, A_{15}}\right)^{x_1(r_1)}(x_4(t_4)x_5(t_5)x_6(t_6))\\
 =& \lambda^{A_{17}, A_{16}^*, A_{15}}(x_1(r_1)\cdot x_4(t_4)x_5(t_5)x_6(t_6)\cdot x_1(r_1)^{-1})\\
 =& \lambda^{A_{17}, A_{16}^*, A_{15}}(x_4(t_4)x_5(t_5+\phi_0(r_1t_4^q))x_6(t_6))\\
 =& \vartheta(A_{17}t_6)
          \cdot \vartheta(A_{16}^*(t_5+\phi_0(r_1t_4^q)))
          \cdot \vartheta\pi_q(A_{15}^qt_4+A_{15}t_4^{q^2})\\
=& \vartheta(A_{17}t_6)
          \cdot \vartheta(A_{16}^*t_5)
          \cdot \vartheta\pi_q(\eta^{-1}((\eta +\eta^{q})A_{15}^{q}+A_{16}^*r_1^{q^2})t_4).
\end{align*}
By Clifford's Theorem, $\tilde{\chi}=\mathrm{Ind}_N^H \lambda^{A_{17}, A_{16}^*, 0}$ and
\begin{align*}
& \mathrm{Res}_N^H\mathrm{Ind}_N^H \lambda^{A_{17}, A_{16}^*, 0}
  =\sum_{r_1\in \mathbb{F}_{q^3}}\left(\lambda^{A_{17}, A_{16}^*, A_{15}}\right)^{x_1(r_1)}
  =\sum_{B_{15}\in \mathbb{F}_{q^3}}\lambda^{A_{17}, A_{16}^*, B_{15}}.
\end{align*}
\end{itemize}
Thus $H$ has $q^7$ linear characters
and $(q-1)q$ irreducible characters of degree $q^3$.
\end{proof}

By the commutator relations, we determine
the conjugacy classes of $H$, and obtain the character table of $H$.
\begin{Corollary}
 Let $x:=x(t_1,0,0,t_4,t_5,t_6)\in H$ be a representative of one conjugacy class of $H$,
 then the character table of $H$ is shown in Table \ref{table:character table-H}.

\begin{table}[!htp]
\caption{Character table of $H$}
\label{table:character table-H}
{
\begin{center}
\begin{tabular}{c|ccc}
$|^Hx|$
& $1$
& $q$
& $q$\\
\begin{tabular}{c}
$x$
\end{tabular}
&  $x_5(t_5)x_6(t_6)$
&  $x_4(t_4^*) x_6(t_6)$
& $x_1(t_1^*)x_4(t_4) x_6(t_6)$
\\\hline
$\tilde{\chi}^{A_{17}, A_{15}, A_{12}}$
& $\vartheta(A_{17}t_6)$
& \begin{tabular}{l}
  $\vartheta(A_{17}t_6)$\\
  $\cdot \vartheta\pi_q(A_{15}^qt_4^*+A_{15}{t_4^*}^{q^2})$
  \end{tabular}
& \begin{tabular}{l}
     $\vartheta(A_{17}t_6)$\\
     $\cdot \vartheta\pi_q(A_{15}^qt_4+A_{15}t_4^{q^2})$\\
     $\cdot \vartheta\pi_q(A_{12}t_1^*)$
   \end{tabular}
\\
$\mathrm{Ind}_N^H\lambda^{A_{17}, A_{16}^*, 0}$
& \begin{tabular}{l}
$q^3\vartheta(A_{17}t_6)$\\
$\cdot\vartheta(A_{16}^*t_5)$
 \end{tabular}
& $0$
& $0$\\
\end{tabular}
\end{center}
}%
\end{table}
\end{Corollary}
Similar to $H$, we construct the irreducible character of the
normal subgroup $X_2X_3X_4X_5X_6$ of $U$.
\begin{Lemma}\label{irr. char. X23456-3D4}
Let $T:=X_2X_3X_4X_5X_6$ and ${\psi}\in \mathrm{Irr}(T)$,
then $T'=X_6$.
\begin{itemize}
 \item [(1)] If $X_6\subseteq \ker{\psi}$,
 let
       $ \bar{H}_{2345}:={{X_6}\backslash T} \cong \bar{X}_2\bar{X}_3\bar{X}_4\bar{X}_5$,
       $ \bar{\chi}^{A_{16}, A_{15},A_{13}, A_{23}}\in \mathrm{Irr}(\bar{H}_{2345})$,
      \begin{align*}
       & \bar{\chi}^{A_{16}, A_{15},A_{13}, A_{23}}
         (\bar{x}_2(t_2)\bar{x}_3(t_3)\bar{x}_4(t_4)\bar{x}_5(t_5))\\
         &  :=\vartheta(A_{16}t_5)
          \cdot \vartheta\pi_q(A_{15}^qt_4+A_{15}t_4^{q^2})
          \cdot \vartheta\pi_q(-A_{13}t_3)
          \cdot \vartheta(A_{23}t_2).
      \end{align*}
Denote by ${\psi}^{A_{16}, A_{15},A_{13}, A_{23}}$ the lift of
 $\bar{\chi}^{A_{16}, A_{15},A_{13}, A_{23}}$ to $T$,
then
\begin{align*}
 \mathrm{Irr}(T)_1:=& \left\{{\psi}\in \mathrm{Irr}(T) {\,}\middle|{\,} X_6\subseteq \ker{\psi} \right\}
 = \left\{{\psi}^{A_{16}, A_{15},A_{13}, A_{23}}
      {\,}\middle|{\,} A_{16},A_{23}\in \mathbb{F}_q, A_{13},A_{15}\in \mathbb{F}_{q^3}\right\}.
 \end{align*}
 \item [(2)] If $X_6\nsubseteq \ker{\psi}$,
let $\psi^{A_{17}^*}:=\mathrm{Ind}_N^T\lambda^{A_{17}^*, 0, 0}$.
Then
 \begin{align*}
 \mathrm{Irr}(T)_2:=& \left\{\psi\in \mathrm{Irr}(T) {\,}\middle|{\,} X_6\nsubseteq \ker\psi \right\}
 =\left\{\psi^{A_{17}^*}
      {\,}\middle|{\,} A_{17}\in \mathbb{F}_q^* \right\}.
 \end{align*}
\end{itemize}
Thus, $\mathrm{Irr}(T)=\mathrm{Irr}(T)_1 \dot{\cup} \mathrm{Irr}(T)_2$, i.e. $T$ has $q^8$ linear characters
and $(q-1)$ irreducible characters of degree $q^4$.
\end{Lemma}

By the commutator relations,
we also determine
the conjugacy classes of $T$
and the character table of $T$.
\begin{Corollary}
 Let $x:=x(0,t_2,t_3,t_4,t_5,t_6)\in T$ be a representative of one conjugacy class of $T$,
 then the character table of $T$ is the one in Table \ref{table:character table-T}.
\begin{table}[!htp]
\caption{Character table of $T$}
\label{table:character table-T}
{\small
\begin{center}
\begin{tabular}{c|ccccc}
$|^Tx|$
& $1$
& $q$
& $q$
& $q$
& $q$\\
\begin{tabular}{c}
$x$
\end{tabular}
&  $x_6(t_6)$
&  $x_5(t_5^*)$
&  $x_4(t_4^*) x_5(t_5)$
&  $x_3(t_3^*)x_4(t_4) x_5(t_5)$
&  \begin{tabular}{l}
    $x_2(t_2^*)x_3(t_3)$\\
    $\cdot x_4(t_4) x_5(t_5)$
   \end{tabular}
\\\hline
${\psi}^{A_{16}, A_{15},A_{13}, A_{23}}$
& $1$
& $\vartheta(A_{16}t_5^*)$
& \begin{tabular}{l}
     $\vartheta(A_{16}t_5)$\\
     $\cdot \vartheta\pi_q(A_{15}^qt_4^*)$\\
     $\cdot \vartheta\pi_q(A_{15}{t_4^*}^{q^2})$\\
   \end{tabular}
& \begin{tabular}{l}
     $\vartheta(A_{16}t_5)$\\
     $\cdot \vartheta\pi_q(A_{15}^qt_4)$\\
     $\cdot \vartheta\pi_q(A_{15}t_4^{q^2})$\\
     $\cdot \vartheta\pi_q(-A_{13}t_3^*)$\\
   \end{tabular}
& \begin{tabular}{l}
     $\vartheta(A_{16}t_5)$\\
     $\cdot \vartheta\pi_q(A_{15}^qt_4)$\\
     $\cdot \vartheta\pi_q(A_{15}t_4^{q^2})$\\
     $\cdot \vartheta\pi_q(-A_{13}t_3)$\\
     $\cdot \vartheta(A_{23}t_2^*)$
   \end{tabular}
\\
$\psi^{A_{17}^*}$
& $q^4\vartheta(A_{17}^*t_6)$
& $0$
& $0$
& $0$
& $0$\\
\end{tabular}
\end{center}
}%
\end{table}
\end{Corollary}

Now we give
the explicit constructions of the irreducible characters of $U$ when $p>2$.
\begin{Proposition}\label{construction of irr. char. of 3D4}
Let $p>2$,
$A_{12},A_{13},A_{15}\in \mathbb{F}_{q^3}$,
and $A_{23},A_{16},A_{17}\in \mathbb{F}_q$.

Let $A_{12}^*,A_{13}^*,A_{15}^* \in \mathbb{F}_{q^3}^*$,
$A_{23}^*,A_{16}^*,A_{17}^*\in \mathbb{F}_q^*$
and $T^{A_{13}^*}$ be a transversal of $A_{13}^*\mathbb{F}_q^+$ in $\mathbb{F}_{q^3}^+$.
Then all of the pairwise orthogonal irreducible characters of $U$ can be constructed as follows:
 \begin{itemize}
  \item [(1)]
  Let
  $\bar{U}:={X_3X_4X_5X_6}\backslash U=\bar{X}_2\bar{X}_1$,
  $\bar{\chi}_{lin}^{A_{12},A_{23}}\in \mathrm{Irr}(\bar{U})$,
  $\bar{\chi}_{lin}^{A_{12},A_{23}}(\bar{x}_2(t_2)\bar{x}_1(t_1))
         :=\vartheta \pi_q(A_{12}t_1)\cdot \vartheta(A_{23}t_2)$,
  and denote by $\chi_{lin}^{A_{12},A_{23}}$
  the lift of $\bar{\chi}_{lin}^{A_{12},A_{23}}$ to $U$,
  then
  \begin{align*}
   \mathfrak{F}_{lin}
    :=& \{\chi\in \mathrm{Irr}(U) \mid X_3X_4X_5X_6\subseteq \ker{\chi}\}
    = \left\{\chi_{lin}^{A_{12},A_{23}} {\,}\middle|{\,} A_{12}\in \mathbb{F}_{q^3}, A_{23}\in \mathbb{F}_{q} \right\}.
  \end{align*}
  \item [(2)]
  Let $\bar{A}_{12}^{A_{13}^*}\in T^{A_{13}^*}$,
  $\bar{U}:={X_4X_5X_6}\backslash U=\bar{X}_2\bar{X}_1\bar{X}_3$,
  $\bar{H}:=\bar{X}_1\bar{X}_3$,
  $\bar{\chi}_{3,q}^{A_{13},{A}_{12}}\in \mathrm{Irr}(\bar{H})$
  with
  $
  \bar{\chi}_{3,q}^{A_{13},{A}_{12}}(\bar{x}_1(t_1)\bar{x}_3(t_3))
         :=\vartheta \pi_q({A}_{12}t_1-A_{13}t_3)$.
 Let $\chi_{3,q}^{A_{13}^*,\bar{A}_{12}^{A_{13}^*}}$
 be the lift of $\mathrm{Ind}_{\bar{H}}^{\bar{U}} \bar{\chi}_{3,q}^{A_{13}^*,\bar{A}_{12}^{A_{13}^*}}$ to $U$,
  then
  \begin{align*}
   \mathfrak{F}_{3}
    :=& \{\chi\in \mathrm{Irr}(U) \mid  X_4X_5X_6\subseteq \ker{\chi},{\ }X_3 \nsubseteq \ker{\chi}\}
    = \left\{\chi_{3,q}^{A_{13}^*,\bar{A}_{12}^{A_{13}^*}}
    {\,}\middle|{\,} A_{13}^*\in \mathbb{F}_{q^3}^*, \bar{A}_{12}^{A_{13}^*}\in T^{A_{13}^*} \right\}.
  \end{align*}
  \item [(3)]
  Let
  $\bar{U}:={X_5X_6}\backslash U=\bar{X}_2\bar{X}_1\bar{X}_3\bar{X}_4$,
  $\bar{H}:=\bar{X}_2\bar{X}_3\bar{X}_4$,
  $\bar{\chi}_{4,q^3}^{A_{15},A_{23},A_{13}}\in \mathrm{Irr}(\bar{H})$,
  and
  \begin{align*}
  \bar{\chi}_{4,q^3}^{A_{15},A_{23},A_{13}}(\bar{x}_2(t_2)\bar{x}_3(t_3)\bar{x}_4(t_4))
         := \vartheta(A_{23}t_2)\cdot \vartheta\pi_q(-A_{13}t_3)
            \cdot \vartheta \pi_q(A_{15}^qt_4+A_{15}t_4^{q^2}).
  \end{align*}
  Denote by $\chi_{4,q^3}^{A_{15}^*,A_{23}}$
  the lift of $\mathrm{Ind}_{\bar{H}}^{\bar{U}} \bar{\chi}_{4,q^3}^{A_{15}^*,A_{23},0}$ to $U$,
  then
  \begin{align*}
   \mathfrak{F}_{4}
    :=& \{\chi\in \mathrm{Irr}(U) \mid  X_5X_6\subseteq \ker{\chi},{\ }X_4 \nsubseteq \ker{\chi}\}
    = \left\{\chi_{4,q^3}^{A_{15}^*,A_{23}}
    {\,}\middle|{\,} A_{15}^*\in \mathbb{F}_{q^3}^*, A_{23}\in \mathbb{F}_{q} \right\}.
  \end{align*}
  \item [(4)]
  Let
  $\bar{U}:={X_6}\backslash U=\bar{X}_2\bar{X}_1\bar{X}_3\bar{X}_4\bar{X}_5$,
  $\bar{H}:=\bar{X}_2\bar{X}_3\bar{X}_4\bar{X}_5$,
  $\bar{\chi}_{5,q^3}^{A_{16},A_{23},A_{13},A_{15}}\in \mathrm{Irr}(\bar{H})$,
  and
 \begin{align*}
  &\bar{\chi}_{5,q^3}^{A_{16},A_{23},A_{13},A_{15}}
           (\bar{x}_2(t_2)\bar{x}_3(t_3)\bar{x}_4(t_4)\bar{x}_5(t_5))\\
         :=& \vartheta(A_{23}t_2)\cdot \vartheta \pi_q(-A_{13}t_3)
           \cdot \vartheta \pi_q(A_{15}^qt_4+A_{15}t_4^{q^2})
           \cdot \vartheta(A_{16}t_5).
\end{align*}
  Denote by $\chi_{5,q^3}^{A_{16}^*,A_{23},A_{13}}$
  the lift of $\mathrm{Ind}_{\bar{H}}^{\bar{U}} \bar{\chi}_{5,q^3}^{A_{16}^*,A_{23},A_{13},0}$ to $U$,
  then
  \begin{align*}
   \mathfrak{F}_{5}
    :=& \{\chi\in \mathrm{Irr}(U) \mid  X_6\subseteq \ker{\chi},{\ }X_5 \nsubseteq \ker{\chi}\}
    = \left\{\chi_{5,q^3}^{A_{16}^*,A_{23},A_{13}}
    {\,}\middle|{\,} A_{16}^*\in \mathbb{F}_{q}^*, A_{23}\in \mathbb{F}_{q}, A_{13}\in \mathbb{F}_{q^3} \right\}.
  \end{align*}
  \item [(5)]
  Let
  $ H:=X_1X_4X_5X_6$,
  $ \bar{H}:={X_4X_5}\backslash H \cong \bar{X}_1\bar{X}_6$,
  $ \bar{\chi}_{6,q^4}^{A_{17},A_{12}}\in \mathrm{Irr}(\bar{H})$,
  and
  \begin{align*}
  \bar{\chi}_{6,q^4}^{A_{17},A_{12}}(\bar{x}_1(t_1)\bar{x}_6(t_6))
    :=\vartheta\pi_q(A_{12}t_1)\cdot \vartheta(A_{17}t_6).
  \end{align*}
Let
  $\tilde{\chi}_{6,q^4}^{A_{17},A_{12}}$
    denote the lift of $\bar{\chi}_{6,q^4}^{A_{17},A_{12}}$
    from $\bar{H}$ to $H$,
  and $\chi_{6,q^4}^{A_{17}^*,A_{12}}:=\mathrm{Ind}_{H}^{U} \tilde{\chi}_{6,q^4}^{A_{17}^*,A_{12}}$,
  then
  \begin{align*}
   \mathfrak{F}_{6}
    :=& \{\chi\in \mathrm{Irr}(U) \mid  X_6 \nsubseteq \ker{\chi}\}
    = \left\{ \chi_{6,q^4}^{A_{17}^*,A_{12}}
    {\,}\middle|{\,} A_{17}^*\in \mathbb{F}_{q}^*, A_{12}\in \mathbb{F}_{q^3} \right\}.
  \end{align*}
 \end{itemize}
Hence $\mathrm{Irr}(U)=\mathfrak{F}_{lin} \dot{\cup} \mathfrak{F}_{3}
                        \dot{\cup} \mathfrak{F}_{4} \dot{\cup} \mathfrak{F}_{5} \dot{\cup} \mathfrak{F}_{6}$.
\end{Proposition}
\begin{proof}
Let $\chi \in \mathrm{Irr}(U)$.
\begin{itemize}
\item [(1)]  {\it Family $\mathfrak{F}_{lin}$, where $X_3X_4X_5X_6 \subseteq \ker\chi$}.

Since $U'=X_3X_4X_5X_6$,
all linear characters of $U$
are precisely the lifts of the linear characters of
the quotient group ${X_3X_4X_5X_6}\backslash U$.
\item [(2)]
{\it Family $\mathfrak{F}_{3}$, where $X_4X_5X_6 \subseteq \ker\chi$
              and $X_3\nsubseteq \ker\chi$}.

Let $\bar{U}:={X_4X_5X_6}\backslash U=\bar{X}_2\bar{X}_1\bar{X}_3$,
then the commutator relation in $\bar{U}$ is
$[\bar{x}_2(t_2), \bar{x}_1(t_1)]= \bar{x}_3(t_2t_1)$.
Let $\bar{H}:=\bar{X}_1\bar{X}_3$,
$\bar{\chi}_{3,q}^{A_{13},{A}_{12}}\in \mathrm{Irr}(\bar{H})$
and
\begin{align*}
 \bar{\chi}_{3,q}^{A_{13},{A}_{12}}(\bar{x}_1(t_1)\bar{x}_3(t_3))
          :=\vartheta \pi_q({A}_{12}t_1)\cdot \vartheta \pi_q(-A_{13}t_3).
\end{align*}
We note that $\bar{X}_2$ is a transversal of $\bar{H}$ in $\bar{U}$,
and  $Z(\bar{U})=\bar{X}_3$.
For all $s_2\in \mathbb{F}_q$,
\begin{align*}
 & \left(\bar{\chi}_{3,q}^{A_{13}^*,{A}_{12}}\right)^{\bar{x}_2(s_2)}
 (\bar{x}_1(t_1)\bar{x}_3(t_3))
 =\bar{\chi}_{3,q}^{A_{13}^*,{A}_{12}}
   (\bar{x}_2(s_2)\cdot \bar{x}_1(t_1)\bar{x}_3(t_3) \cdot \bar{x}_2(s_2)^{-1})\\
 =& \bar{\chi}_{3,q}^{A_{13}^*,{A}_{12}}
     (\bar{x}_1(t_1)\bar{x}_3(t_3+s_2t_1))
 =\vartheta \pi_q({A}_{12}t_1-A_{13}^*(t_3+s_2t_1))\\
 =&\vartheta \pi_q(({A}_{12}-s_2A_{13}^*)t_1-A_{13}^*t_3)
 =\vartheta \pi_q(({A}_{12}-s_2A_{13}^*)t_1)\cdot \vartheta \pi_q(-A_{13}^*t_3),
\end{align*}
so $I_{\bar{U}}(\bar{\chi}_{3,q}^{A_{13}^*,{A}_{12}})=\bar{H}$.
By Clifford theory,
$\tilde{\chi}_{3,q}^{A_{13}^*,{A}_{12}}:=\mathrm{Ind}_{\bar{H}}^{\bar{U}} \bar{\chi}_{3,q}^{A_{13}^*,{A}_{12}}
\in \mathrm{Irr}(\bar{U})$ and
\begin{align*}
 \mathrm{Res}^{\bar{U}}_{\bar{H}}
 \tilde{\chi}_{3,q}^{A_{13}^*,{A}_{12}}
 =\sum_{s_2\in \mathbb{F}_q}{\left(\bar{\chi}_{3,q}^{A_{13}^*,{A}_{12}}\right)^{\bar{x}_2(s_2)}}
 =\sum_{s_2\in \mathbb{F}_q}{\bar{\chi}_{3,q}^{A_{13}^*,{A}_{12}-s_2A_{13}^*}}.
\end{align*}

By Clifford theory,
there are $q^2(q^3-1)$ almost faithful irreducible characters of $\bar{U}$, i.e.
$
\{ \tilde{\chi}_{3,q}^{A_{13}^*,\bar{A}_{12}^{A_{13}^*}}
 \mid
A_{13}^*\in \mathbb{F}_{q^3}^*, \bar{A}_{12}^{A_{13}^*}\in T^{A_{13}^*}
\}$.
Let ${\chi}_{3,q}^{A_{13}^*,\bar{A}_{12}^{A_{13}^*}}$
be the lift of $\tilde{\chi}_{3,q}^{A_{13}^*,\bar{A}_{12}^{A_{13}^*}}$
to $U$.
Thus,
 \begin{align*}
   \mathfrak{F}_{3}
    =& \{\chi\in \mathrm{Irr}(U) \mid  X_4X_5X_6\subseteq \ker{\chi},{\ }X_3 \nsubseteq \ker{\chi}\}
    = \left\{\chi_{3,q}^{A_{13}^*,\bar{A}_{12}^{A_{13}^*}}
    {\,}\middle|{\,} A_{13}^*\in \mathbb{F}_{q^3}^*, \bar{A}_{12}^{A_{13}^*}\in T^{A_{13}^*} \right\}.
  \end{align*}
\item[(3)]
The proof of the Family $\mathfrak{F}_{4}$ is an adaptation of the proof of the Family $\mathfrak{F}_{3}$.
\item [(4)]
The proof of the Family $\mathfrak{F}_{5}$ is also an adaptation of the proof of the Family $\mathfrak{F}_{3}$.
\item [(5)]
{\it Family $\mathfrak{F}_{6}$, where $X_6\nsubseteq \ker\chi$}.

Let $T:=X_2X_3X_4X_5X_6$,  $N:=X_4X_5X_6$,
and $\chi\in \mathrm{Irr}(U)$ such that $X_6\nsubseteq \ker(\chi)$,
then $Z(T)=Z(U)=X_6$.
Let $\psi\in \mathrm{Irr}(T)$ and
$\langle \psi, \mathrm{Res}^U_T \chi \rangle_T >0$,
then $X_6\nsubseteq \ker \psi$ by \ref{restriction center-3D4}.
Let $\lambda^{A_{17}, A_{16}, A_{15}}\in \mathrm{Irr}(N)$
and
 \begin{align*}
  &\lambda^{A_{17}, A_{16}, A_{15}}(x_4(t_4)x_5(t_5)x_6(t_6))
         :=\vartheta(A_{17}t_6)\cdot \vartheta(A_{16}t_5)
          \cdot \vartheta\pi_q(A_{15}^qt_4+A_{15}t_4^{q^2}),
\end{align*}
then
by \ref{irr. char. X23456-3D4},
$\{\psi\in \mathrm{Irr}(T) \mid  X_6 \nsubseteq \ker{\psi}\}
    = \{\mathrm{Ind}_N^T{\lambda^{A_{17}^*,0,0}} \mid A_{17}^*\in \mathbb{F}_{q}^*\}
    = \{\psi^{A_{17}^*} \mid A_{17}^*\in \mathbb{F}_{q}^*\}$.
By (5) of \ref{some inertia groups N-3D4}, $I_U(\psi^{A_{17}^*})=U$,
so $\mathrm{Res}^{U}_T{\chi}=z^*\psi^{A_{17}^*}$ for some $z^*\in \mathbb{N}^*$.
Thus
\begin{align*}
 \mathfrak{F}_6
=& \{\chi\in \mathrm{Irr}(U) \mid X_6 \nsubseteq \ker\chi\}
= \bigcup _{\substack{\psi\in \mathrm{Irr}(T)\\ X_6 \nsubseteq \ker\psi}}
   \{\chi\in \mathrm{Irr}(U) \mid  \langle \chi, \mathrm{Ind}_T^U\psi \rangle_U>0 \}\\
{=}
 & \bigcup _{A_{17}^*\in \mathbb{F}_q^*}\{\chi\in \mathrm{Irr}(U)
     \mid   \langle \chi ,
                \mathrm{Ind}_T^U{\psi^{A_{17}^*}} \rangle_U>0\}\\
{=}
 & \bigcup _{A_{17}^*\in \mathbb{F}_q^*}\{\chi\in \mathrm{Irr}(U)
     \mid   \langle \chi , \mathrm{Ind}_N^U{\lambda^{A_{17}^*,0,0}}\rangle_U>0\}.
\end{align*}
Let $H:={X}_1{X}_4{X}_5{X}_6$,
then $H'=X_5$ and $Z(H)=X_4X_5\trianglelefteq H$.
Let $\tilde{\chi}^{A_{17},A_{15},A_{12}}\in \mathrm{Irr}(H)$
as in (1) of
\ref{irr. char. X1456-3D4}.
For all $x_4(t_4)x_5(t_5)x_6(t_6)\in N$,
\begin{align*}
   & \left(\mathrm{Res}^H_N \tilde{\chi}^{A_{17}^*,0,A_{12}}\right)(x_4(t_4)x_5(t_5)x_6(t_6))
 =  \tilde{\chi}^{A_{17}^*,0,A_{12}}(x_4(t_4)x_5(t_5)x_6(t_6))\\
 =&  \bar{\chi}(\bar{x}_4(t_4)\bar{x}_6(t_6))
 = \vartheta(A_{17}^*t_6)
 = \lambda^{A_{17}^*,0,0}(x_4(t_4)x_5(t_5)x_6(t_6)).
\end{align*}
Thus $\mathrm{Res}^H_N {\tilde{\chi}^{A_{17}^*,0,A_{12}}}=\lambda^{A_{17}^*,0,0}$
   { for all } $A_{12}\in \mathbb{F}_{q^3}$.
By (4) of \ref{some inertia groups N-3D4}
   we get    $I_H(\lambda^{A_{17}^*,0,0})=H$,
thus $\mathrm{Ind}_N^H{\lambda^{A_{17}^*,0,0}}
         =\sum_{A_{12}\in \mathbb{F}_{q^3}}{\tilde{\chi}^{A_{17}^*,0,A_{12}}} $.
By (2) of \ref{some inertia groups N-3D4}
       $I_U(\lambda^{A_{17}^*,0,0})=H$.
Then by Clifford's Theorem, $\mathrm{Ind}_H^U \tilde{\chi}^{A_{17}^*,0,A_{12}}\in \mathrm{Irr}(U)$
for all $A_{17}^*\in \mathbb{F}_q^*$.
Thus
\begin{align*}
 \mathfrak{F}_6
{=}& \bigcup _{A_{17}^*\in \mathbb{F}_q^*}\{\chi\in \mathrm{Irr}(U)
     \mid   \langle \chi ,
                      \mathrm{Ind}_H^U\mathrm{Ind}_N^H{\lambda^{A_{17}^*,0,0}}\rangle_U>0\}\\
{=}& \bigcup _{\substack{A_{17}^*\in \mathbb{F}_q^*\\A_{12}\in \mathbb{F}_{q^3}}}
      \{\chi\in \mathrm{Irr}(U)
     \mid   \langle \chi , \mathrm{Ind}_H^U{\tilde{\chi}^{A_{17}^*,0,A_{12}}}\rangle_U>0\}\\
{=}& \left\{\mathrm{Ind}_H^U{\tilde{\chi}^{A_{17}^*,0,A_{12}}}
    {\,}\middle|{\,}  A_{17}^*\in \mathbb{F}_q^*,A_{12}\in \mathbb{F}_{q^3} \right\}.
\end{align*}
For $A_{17}^*\in \mathbb{F}_q^*$ and $A_{12}\in \mathbb{F}_{q^3}$,
$X_4X_5\subseteq \ker(\tilde{\chi}^{A_{17}^*,0,A_{12}})$ and $X_4X_5 \trianglelefteq H$,
so $\tilde{\chi}^{A_{17}^*,0,A_{12}}$ is the lift to $H$ of some irreducible character of
 $\bar{H}:={{X_4X_5}\backslash H}\cong \bar{X}_1\bar{X}_6$.
Let
$\bar{\chi}_{6,q^4}^{A_{17}^*,A_{12}}\in \mathrm{Irr}(\bar{H})$,
and
$\bar{\chi}_{6,q^4}^{A_{17}^*,A_{12}}(\bar{x}_1(t_1)\bar{x}_6(t_6))
    :=\vartheta\pi_q(A_{12}t_1)\cdot \vartheta(A_{17}^*t_6)$.
Let  $\tilde{\chi}_{6,q^4}^{A_{17}^*,A_{12}}$
    denote the lift of $\bar{\chi}_{6,q^4}^{A_{17}^*,A_{12}}$
    from $\bar{H}$ to $H$,
then $\tilde{\chi}_{6,q^4}^{A_{17}^*,A_{12}}=\tilde{\chi}^{A_{17}^*,0,A_{12}}$.
Let ${\chi}_{6,q^4}^{A_{17}^*,A_{12}}:=\mathrm{Ind}_H^U{\tilde{\chi}_{6,q^4}^{A_{17}^*,A_{12}}}$,
then
\begin{align*}
 \mathfrak{F}_6
{=}& \left\{\mathrm{Ind}_H^U{\tilde{\chi}^{A_{17}^*,0,A_{12}}}
    {\,}\middle|{\,}  A_{17}^*\in \mathbb{F}_q^*,A_{12}\in \mathbb{F}_{q^3} \right\}
{=} \left\{{\chi}_{6,q^4}^{A_{17}^*,A_{12}}
       {\,}\middle|{\,} A_{17}^*\in \mathbb{F}_q^*, A_{12}\in \mathbb{F}_{q^3} \right\}.
\end{align*}
\end{itemize}
\end{proof}

\begin{Remark}
From the proof (5) of Propositon \ref{construction of irr. char. of 3D4},
we obtain that
\begin{align*}
 & \mathrm{Ind}_T^U \psi^{A_{17}^*}
=\mathrm{Ind}_N^U{\lambda^{A_{17}^*,0,0}}
=\mathrm{Ind}_H^U \bigg({\sum_{A_{12}\in \mathbb{F}_q}{\tilde{\chi}^{A_{17}^*,0,A_{12}}}}\bigg)\\
=& \mathrm{Ind}_H^U \bigg({\sum_{A_{12}\in \mathbb{F}_q}{\tilde{\chi}_{6,q^4}^{A_{17}^*,A_{12}}}}\bigg)
=\sum_{A_{12}\in \mathbb{F}_q} {\mathrm{Ind}_H^U{\tilde{\chi}_{6,q^4}^{A_{17}^*,A_{12}}}}
=\sum_{A_{12}\in \mathbb{F}_q}{{\chi}_{6,q^4}^{A_{17}^*,A_{12}}}
\end{align*}
\end{Remark}

\begin{Remark}
\label{irr. characters p=2}
The irreducible characters in the family $\mathfrak{F}_4$ in Proposition \ref{construction of irr. char. of 3D4}
are different for $p=2$, and the construction of these irreducible characters for can be found in  \cite[3.3]{Le}.
The irreducible characters in the other families are the same as these for $p>2$.
\end{Remark}


\section{Character table}
\label{sec:character table}
In this section, we determine the character table of
the Sylow $p$-subgroup $U$ of ${^3}D_4(q^3)$ when $p>2$.
We use the notations as in Section \ref{sec:irreducible characters}.
\begin{Proposition}
\label{prop:character table-3D4}
The character table of $U$ for $p>2$
is the one in Table \ref{table:character table-3D4}.
\end{Proposition}

\begin{sidewaystable}
\caption{Character table of $U$ for $p>2$}
\label{table:character table-3D4}
{
\footnotesize
\begin{align*}
\renewcommand\arraystretch{1.5}
\begin{array}{l|ccccccccc}
×
& I_8
& \begin{array}{c}
  x_1(t_1^*)\\
  \cdot x_6(t_6)
  \end{array}
&
\begin{array}{c}
  x_1(t_1^*)\\
  \cdot x_3(\bar{t}_3^*)
\end{array}
&
\begin{array}{c}
 x_2(t_2^*)\\
 \cdot x_4(t_4)\\
 \cdot x_5(t_5)
\end{array}
&
\begin{array}{c}
 x_2(t_2^*)\\
 \cdot x_1(t_1^*)
\end{array}
&
\begin{array}{c}
x_3(t_3^*)\\
\cdot x_5(t_5)
\end{array}
& x_4(t_4^*)
& x_5(t_5^*)
& x_6(t_6^*)\\
\hline
\chi_{lin}^{0,0}
& 1
& 1
& 1
& 1 & 1 & 1 & 1 & 1 & 1\\
\chi_{lin}^{A_{12}^*,0}
& 1
& \vartheta \pi_q (A_{12}^*t_1^*)
& \vartheta \pi_q (A_{12}^*t_1^*)
& 1
& \vartheta \pi_q (A_{12}^*t_1^*)
& 1 & 1 & 1 & 1 \\
\chi_{lin}^{0,A_{23}^*}
& 1
& 1
& 1
& \vartheta (A_{23}^*t_2^*)
& \vartheta (A_{23}^*t_2^*)
& 1 & 1 & 1 & 1 \\
\chi_{lin}^{A_{12}^*,A_{23}^*}
& 1
& \vartheta \pi_q (A_{12}^*t_1^*)
& \vartheta \pi_q (A_{12}^*t_1^*)
& \vartheta (A_{23}^*t_2^*)
& \begin{array}{l}
\vartheta \pi_q (A_{12}^*t_1^*)\\
\cdot \vartheta (A_{23}^*t_2^*)
\end{array}
& 1 & 1 & 1 & 1
\\
\chi_{3,q}^{A_{13}^*,\bar{A}_{12}^{A_{13}^*}}
& q
& \divideontimes
& \divideontimes
& 0
& 0
& \begin{array}{l}
  \vartheta\pi_q(-A_{13}^*t_{3}^*)\\
  \cdot q
  \end{array}
& q & q & q
\\
\chi_{4,q^3}^{A_{15}^*,A_{23}}
& q^3
& 0
& 0
& \divideontimes
& 0
& 0
& \begin{array}{l}
\vartheta\pi_q({A_{15}^*}^qt_4^*)\\
\cdot \vartheta\pi_q({A_{15}^*}{t_4^*}^{q^2})\\
\cdot  q^3
  \end{array}
& q^3 & q^3
\\
\chi_{5,q^3}^{A_{16}^*,A_{23},A_{13}}
& q^{3}
& 0
& 0
& \divideontimes
& 0
& \divideontimes
& 0
& \begin{array}{l}
  \vartheta(A_{16}^*t_{5}^*)\\
  \cdot q^{3}
  \end{array}
& q^{3}
\\
\chi_{6,q^4}^{A_{17}^*,A_{12}}
& q^{4}
& \divideontimes
& 0
& 0 & 0
& 0 & 0 & 0
& \begin{array}{l}
  \vartheta(A_{17}^*t_{6}^*)\\
  \cdot q^{4}
  \end{array}
\end{array}
\\
\end{align*}
}
where
 the elements of the 1st column (i.e. the row headers)
 are the complete pairwise orthogonal
  irreducible characters of $U$,
and more details can be seen in Proposition \ref{construction of irr. char. of 3D4},
 e.g. $\bar{A}_{12}^{A_{13}^*}\in T^{A_{13}^*}$.
 The entries of the 1st row (i.e. the column headers)
 are all of the representatives of conjugacy classes of $U$,
and more details can be seen in Proposition \ref{prop:conjugacy classes-3D4},
e.g. $0\neq {\bar{t}}_3^*\in T^{t_1^*}$.
The 6 $\divideontimes$-values  in the table are followed:
\end{sidewaystable}
\newpage
 \leftline{The (9,3)-entry of Table \ref{table:character table-3D4} is}
 \begin{align*}
  \chi_{6,q^4}^{A_{17}^*,A_{12}}\left(x_1(t_1^*)x_6(t_6)\right)
 =& \sum_{r_3\in \mathbb{F}_{q^3}}{\vartheta \pi_q
    \left( A_{12}t_1^*+A_{17}^*\phi_0( -t_1^*r_3^{q^2+q})+A_{17}^*t_6 \right)}\\
 =& \vartheta \pi_q(A_{12}t_1^*+A_{17}^*t_6)
   \cdot \sum_{r_3\in \mathbb{F}_{q^3}}{\vartheta
    \left( A_{17}^*\phi_0 ( -t_1^*r_3^{q^2+q}) \right)},
\end{align*}
 The (8,5)-entry of Table \ref{table:character table-3D4} is
\begin{align*}
  \chi_{5,q^3}^{A_{16}^*,A_{23},A_{13}}\left(x_3(t_3^*)x_5(t_5) \right)
 =& \sum_{r_1\in \mathbb{F}_{q^3}}{\vartheta \pi_q
    \left( -A_{13}t_3^*+A_{16}^*\phi_0( t_3^*r_1^{q^2+q})+A_{16}^*t_5 \right)}\\
 =& \vartheta \pi_q(-A_{13}t_3^*+A_{16}^*t_5)
   \cdot \sum_{r_1\in \mathbb{F}_{q^3}}{\vartheta
    \left( A_{16}^*\phi_0( t_3^*r_1^{q^2+q}) \right)},
\end{align*}
 The (8,7)-entry of Table \ref{table:character table-3D4} is
\begin{align*}
 & \chi_{5,q^3}^{A_{16}^*,A_{23},A_{13}}\left(x_2(t_2^*)x_4(t_4)x_5(t_5) \right)\\
=& \sum_{r_1\in \mathbb{F}_{q^3}}\vartheta \pi_q
   ( A_{23}t_2^*+A_{13}r_1t_2^*
    +A_{16}^*(\phi_0 \left(r_1t_4^q \right)-t_2^*r_1^{q^2+q+1})+A_{16}^*t_5 )\\
=&\vartheta (A_{23}t_2^*+A_{16}^*t_5)
  \cdot \sum_{r_1\in \mathbb{F}_{q^3}}\vartheta \pi_q
   \left( A_{13}r_1t_2^*
    +A_{16}^*(\phi_0 \left(r_1t_4^q \right)-t_2^*r_1^{q^2+q+1}) \right),
\end{align*}
 The (7,5)-entry of Table \ref{table:character table-3D4} is
\begin{align*}
 &\chi_{4,q^3}^{A_{15}^*,A_{23}}\left(x_2(t_2^*)x_4(t_4)x_5(t_5) \right)\\
=& \sum_{r_1\in \mathbb{F}_{q^3}}\vartheta \pi_q
    ( A_{23}t_2^*
    +{A_{15}^*}^q(t_4-t_2^*r_1^{q+1})
                +A_{15}^*(t_4^{q^2}-t_2^*r_1^{q^2+1}) )\\
=& \vartheta(A_{23}t_2^*)
 \cdot \sum_{r_1\in \mathbb{F}_{q^3}}\vartheta \pi_q
    \left( {A_{15}^*}^q(t_4-t_2^*r_1^{q+1})
                +A_{15}^*(t_4^{q^2}-t_2^*r_1^{q^2+1}) \right),
\end{align*}
The (6,3)-entry of Table \ref{table:character table-3D4} is
\begin{align*}
 &\chi_{3,q}^{A_{13}^*,\bar{A}_{12}^{A_{13}^*}}\left(x_1(t_1^*)x_6(t_6)\right)
=  \sum_{r_2\in \mathbb{F}_{q}}\vartheta \pi_q
    \left( (\bar{A}_{12}^{A_{13}^*}-A_{13}^*r_2)t_1^* \right),
\end{align*}
 The (6,4)-entry of Table \ref{table:character table-3D4} is
\begin{align*}
 &\chi_{3,q}^{A_{13}^*,\bar{A}_{12}^{A_{13}^*}}\left(x_1(t_1^*)x_3(\bar{t}_3^*)\right)
= \sum_{r_2\in \mathbb{F}_{q}}\vartheta \pi_q
    \left( (\bar{A}_{12}^{A_{13}^*}-A_{13}^*r_2)t_1^*-\bar{t}_3^*A_{13}^* \right).
 \end{align*}

\begin{proof}
Let $x:=x(t_1,t_2,t_3,t_4,t_5,t_6)\in U$.
\begin{itemize}
 \item[(1)] {\it Calculate the values of $\chi_{6,q^4}^{A_{17}^*,A_{12}}$
 for all $A_{17}^*\in \mathbb{F}_q^*$ and $A_{12}\in \mathbb{F}_{q^3}$}.

We use the notations of (5) of Proposition \ref{construction of irr. char. of 3D4},
then
\begin{align*}
\chi_{6,q^4}^{A_{17}^*,A_{12}}\left(x\right)
=& \left(\mathrm{Ind}_H^U \tilde{\chi}_{6,q^4}^{A_{17}^*,A_{12}}\right)(x)
= \frac{1}{|H|}
    \sum_{\substack{g\in {U}\\g\cdot x\cdot g^{-1}\in H} }
    \tilde{\chi}_{6,q^4}^{A_{17}^*,A_{12}}(g\cdot x\cdot g^{-1})\\
=& \frac{1}{|H|}
    \sum_{\substack{g\in {U}\\g\cdot x\cdot g^{-1}\in H} }
    \bar{\chi}_{6,q^4}^{A_{17}^*,A_{12}}({X_4X_5}\cdot{(g x g^{-1})}).
\end{align*}
Thus,
\begin{align*}
& \chi_{6,q^4}^{A_{17}^*,A_{12}}\left(x_1(t_1^*)x_3(\bar{t}_3^*)\right)
=\chi_{6,q^4}^{A_{17}^*,A_{12}}\left(x_2(t_2^*)x_4(t_4)x_5(t_5)\right)\\
=& \chi_{6,q^4}^{A_{17}^*,A_{12}}\left(x_2(t_2^*)x_1(t_1^*)\right)
=\chi_{6,q^4}^{A_{17}^*,A_{12}}\left(x_3(t_3^*)x_5(t_5)\right)
\stackrel{gxg^{-1}\notin H}{=}0.
\end{align*}
Then
\begin{align*}
 &\chi_{6,q^4}^{A_{17}^*,A_{12}}\left( x_4(t_4)x_5(t_5)x_6(t_6) \right)\\
=& \frac{1}{|H|}
    \sum_{\substack{g:=x(r_1,r_2,r_3,r_4,r_5,r_6)\in {U}
                  \\g\cdot  x_4(t_4)x_5(t_5)x_6(t_6)\cdot g^{-1}\in H} }
    \tilde{\chi}_{6,q^4}^{A_{17}^*,A_{12}}(g\cdot x_4(t_4)x_5(t_5)x_6(t_6)\cdot g^{-1})\\
{=}
 & \frac{1}{|H|}
    \sum_{\substack{r_1,r_3,r_4\in \mathbb{F}_{q^3}\\r_2,r_5,r_6\in \mathbb{F}_q} }
    \tilde{\chi}_{6,q^4}^{A_{17}^*,A_{12}}
    (x_4(t_4)x_5(t_5+\phi_0\left(r_1t_4^q\right))\\
  & \phantom{\frac{1}{|H|}
    \sum_{\substack{r_1,r_3,r_4\in \mathbb{F}_{q^3}\\r_2,r_5,r_6\in \mathbb{F}_q} }
    \tilde{\chi}_{6,q^4}^{A_{17}^*,A_{12}}
    (}
  \cdot x_6(t_6+r_2t_5+\phi_0(r_1r_2t_4^q)+\phi_0(r_3t_4^q)))\\
= & \frac{1}{|H|}
    \sum_{\substack{r_1,r_3,r_4\in \mathbb{F}_{q^3}\\r_2,r_5,r_6\in \mathbb{F}_q} }
    \bar{\chi}_{6,q^4}^{A_{17}^*,A_{12}}
    (\bar{x}_6\left(t_6+r_2t_5+\phi_0(r_1r_2t_4^q)+\phi_0(r_3t_4^q)\right))\\
= & \frac{1}{q^3}
    \sum_{\substack{r_1,r_3\in \mathbb{F}_{q^3}\\r_2\in \mathbb{F}_q} }
    \bar{\chi}_{6,q^4}^{A_{17}^*,A_{12}}
    (\bar{x}_6\left(t_6+r_2t_5+\phi_0(r_1r_2t_4^q)+\phi_0(r_3t_4^q)\right)),
\end{align*}
so,
\begin{align*}
& \chi_{6,q^4}^{A_{17}^*,A_{12}}(I_8)=q^4,{\ }
 \chi_{6,q^4}^{A_{17}^*,A_{12}}(x_4(t_4^*))
 =\chi_{6,q^4}^{A_{17}^*,A_{12}}(x_5(t_5^*))
 =0,
 {\ }
 \chi_{6,q^4}^{A_{17}^*,A_{12}}(x_6(t_6^*))
  =q^4\cdot \vartheta(A_{17}^*t_6^*).
\end{align*}
Then
\begin{align*}
 \chi_{6,q^4}^{A_{17}^*,A_{12}}\left( x_1(t_1^*)x_6(t_6) \right)
=& \frac{1}{|H|}
    \sum_{\substack{g\in {U}\\g\cdot x_1(t_1^*)x_6(t_6)\cdot g^{-1}\in H} }
    \tilde{\chi}_{6,q^4}^{A_{17}^*,A_{12}}(g\cdot x_1(t_1^*)x_6(t_6)\cdot g^{-1})\\
\stackrel{g:=x(r_1,r_2,r_3,r_4,r_5,r_6)}{=}
 & \frac{1}{|H|}
    \sum_{\substack{r_2=0\\r_1,r_3,r_4\in \mathbb{F}_{q^3}\\r_5,r_6\in \mathbb{F}_q} }
    \bar{\chi}_{6,q^4}^{A_{17}^*,A_{12}}
    (\bar{x}_1(t_1^*)
     \bar{x}_6(t_6+\phi_0\left(-t_1^*r_3^{q^2+q}\right)))\\
= & \sum_{r_3\in \mathbb{F}_{q^3}}
    \bar{\chi}_{6,q^4}^{A_{17}^*,A_{12}}
    (\bar{x}_1(t_1^*)
     \bar{x}_6(t_6+\phi_0\left(-t_1^*r_3^{q^2+q}\right)))\\
= & \vartheta \pi_q(A_{12}t_1^*+A_{17}^*t_6)
   \cdot \sum_{r_3\in \mathbb{F}_{q^3}}{\vartheta
    \left( A_{17}^*\phi_0 ( -t_1^*r_3^{q^2+q}) \right)}.
\end{align*}
Thus all the values of $\chi_{6,q^4}^{A_{17}^*,A_{12}}$ are obtained.
 \item[(2)] {\it Calculate the values of  $\chi_{3,q}^{A_{13}^*,\bar{A}_{12}^{A_{13}^*}}$
 for all $A_{13}^*\in \mathbb{F}_{q^3}^*$
 and $\bar{A}_{12}^{A_{13}^*} \in T^{A_{13}^*}$}.

We use the notations of (2) of Proposition \ref{construction of irr. char. of 3D4}.
For $x=x(t_1,t_2^*,t_3,t_4,t_5,t_6)\in U$ with $0 \neq t_2^*\in \mathbb{F}_q^*$, we have
\begin{align*}
 \chi_{3,q}^{A_{13}^*,\bar{A}_{12}^{A_{13}^*}}(x)
=& \left(\mathrm{Ind}_{\bar{H}}^{\bar{U}} \bar{\chi}_{3,q}^{A_{13}^*,\bar{A}_{12}^{A_{13}^*}}\right)
   (\bar{x})
= \frac{1}{|\bar{H}|}
  \sum_{\substack{\bar{u}\in\bar{U}
                    \\ \bar{u}\cdot \bar{x}\cdot \bar{u}^{-1} \in \bar{H}}}
  \bar{\chi}_{3,q}^{A_{13}^*,\bar{A}_{12}^{A_{13}^*}}
                (\bar{u}\cdot \bar{x}\cdot \bar{u}^{-1})
\stackrel{\bar{u}\cdot \bar{x}\cdot \bar{u}^{-1} \notin \bar{H}}{=}0 .
\end{align*}
If $x=x(t_1,0,t_3,t_4,t_5,t_6)\in U$, we have
\begin{align*}
 &\chi_{3,q}^{A_{13}^*,\bar{A}_{12}^{A_{13}^*}}\big(x(t_1,0,t_3,t_4,t_5,t_6)\big)
= \left(\mathrm{Ind}_{\bar{H}}^{\bar{U}} \bar{\chi}_{3,q}^{A_{13}^*,\bar{A}_{12}^{A_{13}^*}}\right)
   \big(\bar{x}_1(t_1)\bar{x}_{3}(t_3)\big)\\
=& \frac{1}{|\bar{H}|}
  \sum_{\substack{\bar{u}:=\bar{x}_2(r_2)\bar{x}_1(r_1)\bar{x}_3(r_3)\in\bar{U}
                    \\ \bar{u}\cdot \bar{x}_1(t_1)\bar{x}_{3}(t_3)\cdot \bar{u}^{-1} \in \bar{H}}}
  \bar{\chi}_{3,q}^{A_{13}^*,\bar{A}_{12}^{A_{13}^*}}
                (\bar{u}\cdot \bar{x}_1(t_1)\bar{x}_{3}(t_3)\cdot \bar{u}^{-1})\\
=
     & \frac{1}{|\bar{H}|}\sum_{\substack{r_1,r_3\in\mathbb{F}_{q^3}\\r_2\in\mathbb{F}_{q}}}
     \bar{\chi}_{3,q}^{A_{13}^*,\bar{A}_{12}^{A_{13}^*}}(\bar{x}_1(t_1)\bar{x}_{3}(t_3+r_2t_1))\\
=& \sum_{r_2\in\mathbb{F}_{q}}
   \bar{\chi}_{3,q}^{A_{13}^*,\bar{A}_{12}^{A_{13}^*}}(\bar{x}_1(t_1)\bar{x}_{3}(t_3+r_2t_1))
= \sum_{r_2\in \mathbb{F}_{q}}\vartheta \pi_q
    \left(\bar{A}_{12}^{A_{13}^*}t_1-A_{13}^*({t}_3+r_2t_1)\right).
\end{align*}
Thus we get all the values of $\chi_{3,q}^{A_{13}^*,\bar{A}_{12}^{A_{13}^*}}$.
\end{itemize}

Similarly, we calculate all of the other values of the character table.
\end{proof}

In the paper, $p$ is an odd prime.
By Propositions \ref{prop:conjugacy classes-3D4},
\ref{construction of irr. char. of 3D4}
and \ref{prop:character table-3D4},
we obtain the number of the conjugacy classes,
and determine the numbers of the complex irreducible characters of degree $q^c$ with $c\in \mathbb{N}$.
T. Le \cite{Le} counted the number of irreducible characters for all primes $p$.
\begin{Corollary}\label{Higman conjecture-3D4}
Let $\#\mathrm{Irr}$ be the number of irreducible characters of the Sylow $p$-subgroup $U$ of ${^3}D_4(q^3)$,
and $\#\mathrm{Irr}_c$ be the number of irreducible characters of $U$
of dimension $q^c$ with $c\in \mathbb{N}$,
then
{\small
\begin{align*}
\begin{array}{lll}
\#\mathrm{Irr}_4  &= q^4-q^3 &= (q-1)^4+3(q-1)^3+3(q-1)^2+(q-1),\\
\#\mathrm{Irr}_3  &= q^5-q   &= (q-1)^5+5(q-1)^4+10(q-1)^3+10(q-1)^2+4(q-1),\\
\#\mathrm{Irr}_1  &= q^5-q^2 &= (q-1)^5+5(q-1)^4+10(q-1)^3+9(q-1)^2+3(q-1),\\
\#\mathrm{Irr}_0  &= q^4 &= (q-1)^4+4(q-1)^3+6(q-1)^2+4(q-1)+1,
\end{array}
\end{align*}
}%
and
\begin{align*}
    \#\mathrm{Irr}
 =& \# \{\text{Irreducible Characters of } U  \}\\
 =& \# \{\text{Conjugacy Classes of } U\}\\
 =& 2q^5+2q^4-q^3-q^2-q\\
 =& 2(q-1)^5+12(q-1)^4+27(q-1)^3+28(q-1)^2+12(q-1)+1.
\end{align*}
\end{Corollary}

\begin{Remark}
We also consider the analogue of Higman's conjecture,
Lehrer's conjecture and Isaacs' conjecture of $U_n(q)$
for $U$.
By \ref{Higman conjecture-3D4}, the conjectures are true for $U$.
\end{Remark}


\section*{Acknowledgements}
This paper is part of my PhD thesis \cite{sunphd} at the University of Stuttgart, Germany,
so I am deeply grateful to my supervisor Richard Dipper.
I also would like to thank  Markus Jedlitschky,
Mathias Werth
and Yichuan Yang for the
helpful discussions and valuable suggestions.



\bibliographystyle{alpha}
\bibliography{bibliography3D4irr}


\end{document}